\documentclass[6pt, a4paper, twoside]{article}

\usepackage{parskip}
\usepackage[english]{babel}
\usepackage[utf8]{inputenc}
\usepackage[T1]{fontenc}
\usepackage[margin=3cm]{geometry}

\usepackage{latexsym,amsmath,amsthm,bm,lmodern,titlesec,xcolor,graphicx,setspace,pdfpages,marginnote,caption,array,sectsty,siunitx,eurosym,colortbl,float,wrapfig,vmargin,tikz,hyperref, mathtools, enumitem, comment}

\definecolor{Rosso}{RGB}{232,61,61}
\definecolor{Viola}{RGB}{188,30,188}
\definecolor{Celeste}{RGB}{34,186,211}
\definecolor{Arancione}{RGB}{232,181,19}
\definecolor{Blu}{RGB}{0, 44, 92}
\definecolor{Azzurro}{RGB}{57,186,238}

\usepackage{amssymb}

\usepackage{bbm}
\usepackage{subcaption}
\usepackage{makecell}
\usepackage[section]{placeins}
\usepackage{url}
\usepackage{amsthm}
\usepackage{tikz-cd}
\usepackage[capitalise]{cleveref}

\usepackage[nottoc]{tocbibind}

\setpapersize{A4}
\setmarginsrb{20mm}{15mm}{20mm}{30mm}
                 {8mm}{6mm}{10mm}{10mm}

\hypersetup{colorlinks=true,linkcolor=Blu,allbordercolors=white, urlcolor=blue, citecolor=Viola}
\usepackage{enumitem}

\makeatletter
\newcommand\ackname{Acknowledgements}
\if@titlepage
  \newenvironment{acknowledgements}{%
      \titlepage
      \null\vfil
      \@beginparpenalty\@lowpenalty
      \begin{center}%
        \bfseries \ackname
        \@endparpenalty\@M
      \end{center}}%
     {\par\vfil\null\endtitlepage}
\else
  \newenvironment{acknowledgements}{%
      \if@twocolumn
        \section*{\abstractname}%
      \else
        \small
        \begin{center}%
          {\bfseries \ackname\vspace{-.5em}\vspace{\z@}}%
        \end{center}%
        \quotation
      \fi}
      {\if@twocolumn\else\endquotation\fi}
\fi
\makeatother

\theoremstyle{definition}
\newtheorem{definition}{Definition}[section]

\theoremstyle{plain}
\newtheorem{theorem}[definition]{Theorem}
\newtheorem{proposition}[definition]{Proposition}
\newtheorem{corollary}[definition]{Corollary}
\newtheorem{lemma}[definition]{Lemma}

\newtheorem*{proposition*}{Proposition}

\theoremstyle{remark}

\newtheorem*{remark}{Remark}
\newtheorem{example}[definition]{Example}

\newcommand{\textnormalbf}[1]{\textnormal{\textbf{#1}}}
\mathtoolsset{centercolon}
\newcommand{\defeq}{\vcentcolon=}

\newcommand{\vp}{\varphi}

\newcommand{\xra}[1]{\xrightarrow{#1}}

\newcommand{\into}{\hookrightarrow}

\newcommand{\quot}{\twoheadrightarrow}
\newcommand{\co}{\colon} 

\newcommand{\mfr}{\mathfrak}
\newcommand{\mcl}{\mathcal}

\newcommand{\msf}{\mathsf}

\newcommand{\R}{\mathbb R}
\renewcommand{\P}{\mathbb P}
\newcommand{\A}{\mathbb A}
\newcommand{\Z}{\mathbb Z}
\newcommand{\Q}{\mathbb Q}
\newcommand{\C}{\mathbb C}
\newcommand{\N}{\mathbb N}

\renewcommand{\k}{\mathbbm k}

\newcommand{\sub}{\subseteq}

\newcommand{\abs}[1]{\vert #1 \vert}
\newcommand{\sm}{\setminus}

\newcommand{\set}[1]{ \{#1\} }

\renewcommand{\S}{S}
\newcommand{\SO}{\textnormal{SO}}
\newcommand{\tup}[1]{(#1) }
\newcommand{\conf}{\textnormal{Conf} }
\newcommand{\spec}{\textnormal{Spec }}
\newcommand{\sing}{\textnormal{Sing}}

\newcommand{\bu}[1]{\textnormal{Bl}_{#1}\ }

\newcommand{\realbu}[1]{\textnormal{Bl}^\R_{#1}\ }
\newcommand{\logbu}[1]{\textnormal{Bl}^{\log}_{#1}}
\newcommand{\kn}{\textnormal{KN}}
\newcommand{\an}{\textnormal{an}}

\newcommand{\afm}{\textnormal{FM}^{\textnormal{alg}}}
\newcommand{\lfm}{\textnormal{FM}^{\textnormal{log}}}
\newcommand{\afmfun}{\mathcal{X}}

\newcommand{\topfm}{\textnormal{FM}^\textnormal{top}}
\newcommand{\kont}{\textnormal{K}}
\newcommand{\tfun}[1]{\mcl T_{#1}}
\newcommand{\psets}[1]{P_2({#1})}

\newcommand{\tk}{\textnormalbf{FM}}
\newcommand{\gk}{\textnormalbf{CGK}}
\newcommand{\lgk}{\textnormalbf{CGK}^{\log}}
\newcommand{\vgk}{\textnormalbf{CGK}^{\textnormal{V-log}}}
\newcommand{\ld}{\textnormalbf{LD}}

\newcommand{\setcat}{\textnormal{\textbf{Set}}}
\newcommand{\schcat}{\textnormal{\textbf{{Sch}}}}

\newcommand{\vlogcat}{{}^{\textnormal{\textbf{V}}} \textnormal{\textbf{Log}}}
\newcommand{\logcat}{\textnormal{\textbf{{Log}}}}
\newcommand{\vdflogcat}{{}^{\textnormal{\textbf{V}}} \textnormal{\textbf{DF-Log}}}
\newcommand{\dflogcat}{\textnormal{\textbf{{DF-Log}}}}

\title{Log Geometric Models for Little Disks Operads in Even Dimensions}
\author{Oliver Lindström}
\date{}

\begin{document}

\maketitle

\begin{abstract}
    \noindent We construct a model for the (non-unital) $\S^1$-framed little $2d$-dimensional disks operad for any positive integer $d$ using logarithmic geometry. We also show that the unframed little $2d$-dimensional disks operad has a model which can be constructed using log schemes with virtual morphisms.
\end{abstract}

\tableofcontents

\section{Introduction}
A recent construction of Dmitry Vaintrob \cite{vaintrob2021formality,vaintrob2019moduliframedformalcurves}, shows that the framed little $2$-dimensional disks (pseudo)-operad can be modeled as the analytification of a pseudo-operad in log schemes whose underlying schemes are $\overline{\mcl M}_{0,n+1}$, the moduli spaces of stable $(n+1)$-pointed rational curves of genus $0$. See also the expositions of Vaintrob's construction in \cite{bergstrom2023hyperelliptic} and \cite{dupont2024logarithmicmorphismstangentialbasepoints}.

This article gives a generalization of Vaintrob's construction to arbitrary even dimension. For every positive integer $d$, we construct a pseudo-operad $\lgk_d$ in the category of log varieties whose analytification is weakly equivalent to the $2d$-dimensional $\S^1$-framed little disks. We write $\S^1$ for the group $U(1)\cong \SO(2)$ and when we refer to an operad $\textnormalbf{O}$ \emph{framed} by some group $G$ which acts on $\textnormalbf{O}$, what is meant is the semidirect operadic product $\textnormalbf{O} \rtimes G$ \cite{salvatore2003framed}. The underlying spaces of the log varieties in $\lgk_d$ are the moduli spaces for stable $n$-pointed rooted trees of $d$-dimensional projective spaces, introduced by Chen, Gibney, and Krashen in \cite{chen2009pointed} and denoted $T_{d,n}$. The connection to Vaintrob's result is that these moduli spaces are higher dimensional analogs of $\overline{\mcl M}_{0,n+1}$ and in particular we have that $T_{1,n} \cong \overline{\mcl M}_{0,n+1}$.
Specifically, we show that the analytification $\lgk_d$ is homeomorphic to $\tk_{2d} \rtimes \S^1$, the $\S^1$-framed Fulton--MacPherson operad. By a result of Salvatore \cite{salvatore1999configuration} $\tk_n$ is $\SO(n)$-equivariantly weakly equivalent to $\ld_n$, the classical little disks operad in dimension $n$. Thus, this homeomorphism gives a weak equivalence between the analytification of $\lgk_d$ and $\ld_{2d}\rtimes \S^1$.
\theoremstyle{plain}
\newtheorem*{thm: main theorem (non-unital operad iso)}{Theorem \ref{thm: main theorem (non-unital operad iso)}}
\begin{thm: main theorem (non-unital operad iso)}
    The Kato--Nakayama analytification of $\lgk_d$ is homeomorphic to $\tk_{2d}\rtimes \S^1$.
\end{thm: main theorem (non-unital operad iso)} 

We will also show that the individual spaces $\tk_{2d}(n)$ are analytifications of log varieties; however, it will not be possible to construct an operad of log schemes with these as objects whose analytification is $\tk_{2d}$. Nevertheless, an operad with this property, which we denote $\vgk_d$, can be constructed in the category of log schemes with \emph{virtual morphisms}. Virtual morphisms of log schemes were introduced by Howell \cite{howell2017motives} and further studied by Dupont, Panzer, and Pym in a recent article \cite{dupont2024logarithmicmorphismstangentialbasepoints}, in which they also show that $\tk_2$ is homeomorphic to the analytification of an operad in log schemes with virtual morphisms.   
\theoremstyle{plain}
\newtheorem*{thm: virtual log operad}{Theorem \ref{thm: virtual log operad}}
\begin{thm: virtual log operad}
    The Kato--Nakayama analytification of $\vgk_d$ is homeomorphic to $\tk_{2d}$.
\end{thm: virtual log operad}

This has interesting consequences for the cohomology of $\ld_{2d}\rtimes \S^1$ and $\ld_{2d}$. First, $\lgk_d$ and $\vgk_d$ are defined over $\Q$ (and even over $\Z$). This induces a Galois action on the corresponding \'etale cohomology cooperads which lifts to an action on the level of \'etale cochains. If the weights of cohomology are \emph{pure} in a suitable sense the existence of such lifts implies formality of the corresponding operad. See for example Petersen \cite{petersen2013minimalmodelsgtactionformality}. Additionally, using the log geometric structure of these operads one can define mixed Hodge structures on the respective cohomology cooperads. 
Such mixed Hodge structures can also be tools for proving operad formality. In fact, Dupont, Panzer, and Pym use the fact that the Hodge structure on the $k$th cohomology of $\vgk_1$ is pure of weight $2k$ to re-prove that $\ld_2$ is formal.
It should be noted that all log varieties studied in this article are either smooth varieties with log structure associated to a smooth strict normal crossings divisor or smooth, closed, strong deformation retracts thereof with pulled back log structures. Therefore, it is known how one should define mixed Hodge structures for the log varieties studied in this article but in general the theory of Hodge structures in logarithmic geometry is still under development and there is, to the authors knowledge, no explicit construction of a mixed Hodge structure functor on fs log varieties, even if we restrict to the case of Deligne--Faltings log varieties. Naturally, there is even less material on the matter for log varieties with virtual morphisms. Considering this, it is debatable if the results of this article alone proves that there is a mixed Hodge structure on the cohomologies of $\ld_{2d}$ and $\ld_{2d} \rtimes \S^1$ but hopefully it will be clear to the familiar reader how one can define such a structure using our constructions.

From this one may expect that the construction of this article would imply that either of these formality arguments could be used to prove that the $\S^1$-framed little disks are formal in arbitrary even dimension. However, the argument requires, in some sense, that the weights of cohomology are pure and this is no longer the case for $\lgk_d$ when $d\geq 2$. For example, we can create a zig-zag of homotopy equivalences coming from morphisms of log-schemes from $\lgk_d(n)$ to $\conf_n(\A_\C^d)\times \mathbb{G}_m^{n}$ and thus the MHS (mixed Hodge structures) on both these spaces are isomorphic. We know that the MHS on $H^1(\mathbb G_m, \Q)$ is pure of weight $2$ and similarly, the MHS on $H^{2d-1}(\conf_n(\A^d_\C), \Q)$ is pure of weight $2d$. For $d=2$ this implies that the MHS on $H^3(\conf_3(\A_\C^2)\times \mathbb{G}_m^{3},\Q)\cong \Q(-2)\oplus \Q(-3)$ is mixed with weights $4$ and $6$. In fact, the real model for $\ld_n\rtimes \SO(n)$ by Khoroshkin and Willwacher in \cite{khoroshkin2025realmodelsframedlittle} can be used to construct a real model for $\ld_{2m}\rtimes \S^1$ and using this model one can show that $\ld_{2m}\rtimes \S^1$ is not formal.\footnote{This was made clear to the author through comments from Thomas Willwacher.} On the other hand, the (unframed) little disks are famously known to be formal in every dimension \cite{kontsevich1999operads,lambrecths2014formality}.

Even before Vaintrob's article it had long been expected that the little disks, and the framed little disks are in some sense \emph{motivic}, or algebro-geometric. A first clue towards this expectation was the Galois action on $C_*(\ld_2)$ via Grothendieck-Teichmüller theory \cite{drinfeld1990quasitriangular,bar-natan1998associators}. This circle of ideas has been developed by a large number of people, including \cite{morava2003motivicthomisomorphism,petersen2013minimalmodelsgtactionformality,kontsevich1999operads,debrito2024algebrogeometricmodelconfigurationcategory}. 

\subsection{Structure and Results}
Section \ref{sec: intro to DF log-schemes} gives a short introduction of logarithmic geometry and in particular we give a definition of Deligne--Faltings, or DF, log schemes and describe their relationship with ``normal'' log schemes. We also define the real oriented blow-up of a topological space in a section of a vector bundle and we explain how this can be used to give an equivalent definition of the so called Kato--Nakayama analytification \cite{kato1999log} of a DF log scheme. Finally, we use the notion of virtual morphisms of log schemes \cite{howell2017motives, dupont2024logarithmicmorphismstangentialbasepoints} to give a definition of a virtual morphism of DF log schemes and explain why this is interesting. 

In \cref{sec:topological constructions} we recall the definition of the topological Fulton--MacPherson compactification and the Fulton--MacPherson model of the little $D$-dimensional disks operad, denoted $\tk_D$ \cite{kontsevich1994feynman,axelrodsinger1994chernsimons, sinha2004manifold,getzler1994operadshomotopyalgebraiterated}. We describe an action by $\SO(D)$ on $\tk_D$ which is compatible with the corresponding action on $\ld_D$, the $D$-dimensional little disks operad. In the case where $D =2d$ is even this gives an induced action by $\S^1 \cong \SO(2)$ via the diagonal embedding $\SO(2) \into \SO(d)$. This action allows us to define the $\S^1$-framed Fulton--MacPherson operad $\tk_{2d}\rtimes \S^1$ which is weakly equivalent to the $\S^1$-framed little disks operad.

In \cref{sec: geometric operad} we recall the definition of the schemes $T_{d,n}$ introduced by Chen, Gibney, and Krashen \cite{chen2009pointed}.
We use a functor of points description of these schemes to define an operad in schemes with objects $T_{d,n}$ for each fixed $d$. We call this operad the Chen, Gibney, Krashen operad, denoted $\gk_d$.

Building on this, in  \cref{sec: log geometric operad} we define Deligne--Faltings log varieties $\msf T_{d,n}$ whose underlying varieties are $T_{d,n}$ and we extend the symmetry and composition morphisms (but not the unit morphism) of $\gk_d$ to maps of these log varieties. This defines a pseudo-operad of log schemes, $\lgk_d$. We then show that the analytification of $\lgk_d$ is $\tk_{2d}\rtimes \S^1$. We describe how to define the unit morphism of $\lgk_d$ as a virtual morphism of log schemes and furthermore we show that with virtual morphisms of log schemes we can also define an operad $\vgk_d$ whose analytification is $\tk_{2d}$. 

\begin{acknowledgements}
    \noindent This article is a reworked version of the authors master's thesis \cite{lindström2024logarithmic} and is written as a part of the author's Ph.D. studies. The author would like to thank his Ph.D supervisor Dan Petersen for many helpful discussions and suggestions both during the writing of this article and the master's thesis.\\
\end{acknowledgements}

\section{Deligne--Faltings Log Schemes}\label{sec: intro to DF log-schemes}
We will give all definitions and results relating to logarithmic geometry in this article in terms of Deligne--Faltings log structures. The main reasons for this are that some definitions and results are  easier to state using this terminology and, perhaps more importantly, the Kato--Nakayama analytification of a Deligne--Faltings log scheme can be explicitly described in terms of real oriented blow-ups which will turn out to be very useful for the purposes of this article. The precise definition of a DF log scheme may vary between articles but we will use that of Bergstr{\"o}m, Diaconu, Petersen, and Westerland \cite{bergstrom2023hyperelliptic}. In this section we will on some occasions reference ``normal'' log structures and therefore we will spell out the prefix ``DF''. This will not be the case in remaining sections and therefore ``log scheme'' should be taken to mean ``DF log scheme'' outside of \cref{sec: intro to DF log-schemes}.

\begin{definition}
    A \emph{Deligne--Faltings (or DF) log structure} on a scheme $X$ is a finite tuple $\mfr L = \tup{s_i\co \mcl O_X \to \mcl L_i}_{1\leq i \leq n}$ of invertible sheaves with sections. 
    A morphism of log structures on $X$,
    $$\tup{s_i\co \mcl O_X \to \mcl L_i}_{1\leq i \leq n} \to \tup{t_j\co \mcl O_X \to \mcl M_j}_{1\leq j \leq m}$$
    is a collection $\{e_{ij}\}$ of non-negative integers together with $n$ isomorphisms of line bundles
    $$\mcl L_i \xra{\cong} \bigotimes_{1\leq j \leq m}\mcl M_j^{\otimes e_{ij}}$$
    which also identify the sections $s_i$ to the corresponding sections $\bigotimes_{1\leq j \leq m}t_j^{\otimes e_{ij}}$.
\end{definition}
\begin{remark}
    For those with a sufficient background in log geometry: A DF log structure $\mfr L = \tup{s_i\co \mcl O_X \to \mcl L_i}_{1\leq i \leq n}$ on some scheme $X$ induces a log structure $M(\mfr L) = \mcl M^{\log} \to \mcl O_X$ where $\mcl M^{\log}$ is the logification of the pre-log structure 
    $$\mcl M(U) = \bigsqcup_{\alpha_i \geq 0} ((\mcl L_1^\vee)^{\otimes \alpha_1})^* \otimes \dots \otimes \mcl ((\mcl L_n^{\vee})^{\otimes \alpha_n})^*$$
    with the obvious monoid structure and with the map to $\mcl O_X$ induced by $s_i^\vee\co \mcl L_i^\vee \to \mcl O_X$. The superscript $*$ here indicates the subsheaf of invertible sections. A morphism of DF log structures induces a map of corresponding log structures in the obvious, functorial, way.
\end{remark}

\begin{definition}
    A \emph{DF log scheme} $\msf X = (X,\mfr L)$ is a scheme $X$ with a DF log structure $\mfr L$. We call $X$ the \emph{underlying scheme} of $\msf X$, sometimes denoted $\underline{\msf X}$. 
    A morphism of DF log schemes $(X,\mfr L) \to (Y,\mfr M)$ is a morphism of schemes $f\co X\to Y$ and a morphism of DF log structures $f^*\mfr M \to \mfr L$. 
    Such a morphism is said to be \emph{strict} if $f^*\mfr M\to \mfr L$ is an isomorphism of DF log structures.
\end{definition}
An important special case of a DF log scheme is a scheme with the trivial DF log structure, i.e. no line bundles. For a scheme $X$ we call $(X, \emptyset)$ the \emph{log scheme associated to $X$} and we will, by abuse of notation, simply denote this by $X$. If $\msf Y$ is a log scheme and $f\co \underline{\msf Y} \to X$ is a map of schemes then there is a unique morphism of log schemes $\msf Y \to X$ with underlying map $f$. Similarly a morphism $f\co X \to \underline{\msf Y}$ gives a morphism of log schemes $X \to \msf Y$ if and only if the pullbacks via $f$ of all line bundles with sections in the log structure of $\msf Y$ are trivial with a unit section and if this is the case then this morphism of log schemes is unique.

\subsection{Blow-Ups}\label{subsec: blow-ups}
In this section we define real oriented blow-ups in sections of line bundles and prove some results relating to these. These will be used to define the Kato--Nakayama analytification of a log scheme. 

First, let $X$ be a topological space, let $\eta\co E \to X$ be a rank $n$ vector bundle and let $s\co X\to E$ be a section. Furthermore, let $E_0$ denote the image of the zero section and let $E' \defeq E\sm E_0$. Note that the action by $\R_{>0}$ (and in fact the action by $\R^*$) on $E$ restricts to a free action on $E'$. We define the real oriented blow-up of $X$ in $s$, denoted $\realbu{s} X$ as the space
$$\{p\in E'\  \vert\  \exists\  \alpha\in \R_{\geq 0} \co\   p =\alpha \cdot (s\circ \eta)(p)\}/\R_{>0}$$
and we define the blow-down map $\rho\co \realbu{s} X \to X$ as the map induced by passing $p\co E\to X$ to the quotient. Note that $\rho$ has an inverse away from the zero locus of $s$ and for any $x$ in the zero locus of $s$ we have that $\rho^{-1}(x) \cong \S^{n-1}$.

The ``usual'' definition of the real oriented blow-up of a smooth manifold $X$ in a smooth closed sub manifold $Y$, denoted $\realbu{Y}X$ is that it is the complement of a tubular neighbourhood of $Y$ in $X$ and the blow-down map is then taken to be a retraction of this neighbourhood onto $Y$. It is also possible to give a functorial definition as in \cite{Arone2010OnTF}.
The following result, whose proof I will omit, relates this construction to the usual notion of a blow-up in a smooth closed submanifold.
\begin{proposition}
    Let $M$ be a (smooth) manifold, let $E$ be a vector bundle, let $s\co M\to E$ be a smooth section whose image intersects $E_0$ transversally, and let $Y = s^{-1}(E_0)$ be the zero locus of $s$. Then there is a unique isomorphism of spaces over $M$,
    $$\realbu{s} M \cong \realbu{Y} M.$$
\end{proposition}
\begin{remark}
    Note that a necessary, but not sufficient, condition here is that the codimension of $Y$ is the rank of the vector bundle. This resembles the algebraic situation where if $X$ is smooth and $Y$ is a locally complete intersection the blow-up $\bu{Y}X$ is locally a closed subscheme of a projective bundle over $X$.
\end{remark}
Despite these properties one should be aware that $\realbu{s} X$ does not always resemble the ``usual'' definition of a blow-up. For example if we take the blow-up in the zero section of $E$ the result is (up to canonical isomorphism) the sphere bundle of $E$ which, in the case where $X$ is a manifold, has higher dimension than $X$. Even stranger situations are also possible and the blow-up of a manifold is not even necessarily a manifold with corners. For example let $s\co\S^2\to \R^2$, considered as a section of a trivial bundle on $\S^2$, be defined by $s(x,y,z) = (x-1,0)$. The real oriented blow-up $\realbu{s}\S^2$ is the wedge sum $\S^1\vee \S^2$.

\begin{lemma}\label{lemma: pullback of line bundle to complement of zero-section}
    Let $X$ be a topological space, let $\pi\co E\to X$ be a complex line bundle, and let $E'$ be the complement of the zero section in $E$. Then $\pi\vert_{E'}^* E \to E'$ is trivial.
\end{lemma}
\begin{proof}
    The diagonal map $E\to E\times_X E$ restricts to a nowhere $0$ section, $E'\to \pi\vert_{E'}^* E$. Thus $\pi\vert_{E'}^* E$ is a line bundle with a nowhere $0$ section and is hence trivial. 
\end{proof}
\begin{remark}
    If $X$ is a smooth manifold and $E$ is a smooth bundle then $\pi\vert_{E'}^* E \to E'$ is trivial as a smooth bundle by the same argument.
\end{remark}
\begin{corollary}
    Let $s\co X\to E$ be any section.
    The pullback of $E\to X$ to the blow-up $\realbu{s}X$ is a trivial bundle.
\end{corollary}
\begin{proof}
    By introducing a Hermitian metric, the blow-up can be embedded into $E'$ and the blow-down map factors via this embedding. Thus, the pullback of $E$ to this blow-up is trivial by \cref{lemma: pullback of line bundle to complement of zero-section}.
\end{proof}

\begin{corollary}\label{cor: blow up in zero section of tensor prod of blown up line bundles is trivial}
    Let $L_1,L_2,\dots, L_n$ be complex line bundles on a manifold $X$ and let $L=\bigotimes_{i=1}^nL_i^{\otimes e_i}$ where $e_i$ are integers and $\otimes$ is the complex tensor product. Let $\sigma_1,\dots, \sigma_n$ be sections $\sigma_i\co X \to L_i$. Then, there is an isomorphism
    $$\realbu{\tilde \sigma_0}\realbu{\tilde \sigma_n}\dots \realbu{\sigma_1} X \xra{\cong} \left (\realbu{\tilde \sigma_n}\dots \realbu{\sigma_1} X \right )\times \S^1$$
    where $\tilde \sigma_i$ denotes the pullback of $\sigma_i$ through all previous morphisms and $\sigma_0\co X\to L$ is the zero section.
\end{corollary}
\begin{proof}
    This is immediate from the previous corollary.
\end{proof}

\subsection{Kato--Nakayama Analytification}
For a DF log scheme $\msf X =(X, \tup{s_i\co \mcl O_X \to \mcl L_i}_{1\leq i \leq n})$ such that $X$ is of finite type over $\C$ one can construct an associated topological space $\msf X^\kn$ called its \emph{Kato--Nakayama}  analytification. This space is defined as the sequence of real oriented blow-ups
$$\realbu{\tilde s_n}\realbu{\tilde s_{n-1}} \dots \realbu{s_1} X^{\an}$$
where $\tilde s_i$ is the section $s_i$ of the vector bundle $\mcl L_i$ pulled back via all previous blow-ups. This is order independent up to canonical isomorphism since, at each step, we are blowing up in the pullback (i.e. total, not strict, transform) of the corresponding line bundle with section.
Regardless of order we let $\rho_{\msf X}\co \msf X^\kn \to X^\an$ denote the corresponding blow-down map. This construction can be made functorial in the following way.

First, if $f\co (X,\tup{s_i\co \mcl O_X \to \mcl L_i}_{1\leq i \leq n}) \to (Y,\tup{t_i\co \mcl O_X \to \mcl M_i}_{1\leq i \leq n})$ is a strict morphism of log-schemes, i.e a morphism such that the log structure morphism is given by isomorphisms $f^*\mcl M_i \to \mcl L_i^{\otimes 1}$ for each $i$, then we can form the following commutative and cartesian (with respect to the downward vertical arrows) diagram 
\begin{center}
    \begin{tikzcd}
        E_X^1 \arrow[shift left = 0.75ex]{d}\arrow{r} & E_Y^1\arrow[shift left = 0.75ex]{d}\\
        X(\C)\arrow[shift left = 0.75ex]{u}{s}\arrow{r}{f^{\an}} & \arrow[shift left = 0.75ex]{u}{t}Y(\C)
    \end{tikzcd}
\end{center}
where we have $E_X^1 = \mcl L_1^\an$ and $E_Y^1 = \mcl M_1^\an$ and where $s,t$ are the corresponding sections. The blow-ups $\realbu{s_1} X(\C)$ and $\realbu{t_1} Y(\C)$ are subspaces of the quotients by an equivalence relation $\sim$ on $E_X^1$ and $E_Y^1$ respectively where two vectors in a fibre are identified if they are equal up to \textbf{positive} scalar multiplication. By commutativity of the diagram, the subspace corresponding to the blow-up of $X(\C)$ is mapped to the blow-up of $Y(\C)$ as subspaces of  $E_X^1/\sim$ and $E_X^1/\sim$ by the induced quotient space map. Thus, we get an induced subspace map $\realbu{s_1} X(\C) \to \realbu{t_1} Y(\C)$. Iterating this gives the desired map of the entire sequence of blow-ups.

Next, let $\phi \co \mfr M =\tup{t_j\co \mcl O_X \to \mcl M_j}_{1\leq j \leq m} \to \mfr L = \tup{s_i\co \mcl O_X \to \mcl L_i}_{1\leq i \leq n}$ be a morphism of DF log structures on $X$ given by integers $\{e_{ij}\}$ and isomorphisms
$\phi_i \co \mcl M_i \to \bigotimes_{j}\mcl L^{\otimes e_{ij}}$. For simplicity assume that all line bundles are trivial. The general case can be constructed from this by gluing.
Furthermore, let $\pi\co (X, \mfr L)^\kn \to X(\C)$ and $\rho \co (X,\mfr M)^\kn \to X(\C)$ denote the corresponding blow-down maps. Then, by composing with the trivializing isomorphisms, we obtain a unique invertible algebraic function $\lambda_i\co X \to \C$ such that we can identify the analytification of $\phi_i$
with a map $\tilde{\phi_i}\co X(\C)\times\C \to X(\C)\times\bigotimes_{1\leq j\leq n} \C^{\otimes e_{ij}}$ which sends
$$(x,z) \mapsto (x, \lambda_i(x) z\bigotimes_{1\leq j\leq n} 1^{\otimes e_{ij}}).$$
This gives a map of spaces over $X(\C)$, 
$X(\C)\times (\C^*)^n \to X(\C)\times (\C^*)^m$ given by 
$$(x,(z_1,\dots, z_n)) \mapsto (x,(\lambda_1(x)\prod_{j} z_j^{e_{1j}}, \dots, \lambda_n(x)\prod_{j} z_j^{e_{nj}})).$$
By definition of $\lambda_i$ this map sends
$$(x,s_1(x),\dots, s_n(x))\mapsto (x,t_1(x),\dots, t_m(x))$$ 
and hence induces a morphism of blow-ups as quotients of subspaces of these bundles.

Since any morphism of log schemes is the composition of a strict morphism and a morphism given by the identity on underlying spaces these two cases are enough to define the analytification of an arbitrary map of log schemes.

One can verify that this definition of the analytification functor for DF log scheme agrees with the analytification of its associated log scheme originally introduced by Kato and Nakayama in \cite{kato1999log}. This can be checked through direct computation but for a less tedious argument see \cite{bergstrom2023hyperelliptic}. From section \ref{sec:topological constructions} onward we will drop the ``DF'' prefix and let ``log scheme'' mean ``DF log scheme''. All log schemes referenced in this article will indeed be DF log schemes but because of the above the distinction is not important for our purposes anyway.
\begin{remark}
    In \cite{kato1999log}, $\msf X^{\kn}$ is also given an associated sheaf of rings which, provided that $\underline{\msf X}$ is smooth and that the divisor associated to all non zero sections of line bundles is a smooth normal crossings divisor, gives $X^{\kn}$ the structure of a smooth real manifold with corners. This motivates the name ``analytification''. These conditions are all satisfied for the log schemes we consider in this article and we could, with a little more work, replace the word ``homeomorphism'' with ``diffeomorphism'' in \cref{thm: main theorem (non-unital operad iso)} and \cref{thm: virtual log operad}.
\end{remark}

We will not need to use this explicit definition of the analytification of a morphism at any point in this article. Instead, all we will need is that the analytification of a map exists, is functorial, and satisfies the following properties:
\begin{itemize}
    \item For a log scheme $\msf X$ with log structure $\mfr L = \tup{s_i\co \mcl O_{\underline{\msf X}} \to \mcl L_i}_{1\leq i \leq n}$ the analytification of the map $\msf X \to \underline{\msf X}$ given by the identity on underlying schemes is the blow-down map
    $$\msf X^\kn = \realbu{\tilde s_n}\realbu{\tilde s_{n-1}} \dots \realbu{s_1} \underline{\msf X}(\C) \to \underline{\msf X}(\C).$$ 
    \item For a strict morphism of log schemes $f\co \msf X \to \msf Y$ the following is a Cartesian diagram: 
    \begin{center}
        \begin{tikzcd}
            \msf X^\kn \arrow{d}\arrow{r} & \msf Y^\kn \arrow{d}\\
            \underline{\msf{X}}(\C) \arrow{r} & \underline{\msf Y}(\C)
        \end{tikzcd}
    \end{center}
\end{itemize}
Both of these properties are standard and easy to verify so we omit a proof and will use them freely without reference.

\subsection{Virtual Logarithmic Geometry}
In this section we will give a short introduction to virtual morphisms of log schemes which were introduced by Howell \cite{howell2017motives} and further studied by Dupont, Panzer, and Pym \cite{dupont2024logarithmicmorphismstangentialbasepoints}. Such virtual morphisms will be irrelevant for most sections of this article but in \cref{subsec: virtual log operads} we include a discussion on how one can generalize some of our results using virtual morphisms in ways that are not possible otherwise. In the language of normal, i.e. non DF, logarithmic geometry a log structure on a scheme $X$ is a sheaf of monoids $\mcl M$ with a morphism of sheaves of monoids $\alpha\co \mcl M \to \mcl O_X$ which restricts to an isomorphism $\alpha^{-1}(\mcl O_X^\times) \to \mcl O_X^\times$. A map of log structures is then a map of sheaves of monoids over $\mcl O_X$, $\mcl M_1 \to \mcl M_2$. Note that such a map commutes with the maps from $\mcl O^\times$ by definition.

A virtual morphism of log structures is a morphism of the groupifications $\mcl M_1^{\textnormal{gp}} \to \mcl M_2^{\textnormal{gp}}$ which makes the diagram
\begin{center}
    \begin{tikzcd}
        &\mcl O_X^\times\arrow[rd]\arrow[ld]&\\
        \mcl M_1^{\textnormal{gp}}\arrow[rr] &&\mcl M_2^{\textnormal{gp}}
    \end{tikzcd}
\end{center}
commute. Note that the groupifications $\mcl M_1^{\textnormal{gp}}, \mcl M_2^{\textnormal{gp}}$ do not have maps to $\mcl O_X$ which is why we use this weaker condition. Also note that the category of virtual log structures has more morphisms but we define it to have the same objects as the category of log structures. For a short motivation regarding why it makes sense to introduce such virtual morphisms it is worth mentioning that the log Betti and log de Rham cohomology functors are well defined for log schemes with virtual morphisms and the same is also true for the Kato--Nakayama analytification functor.

Although Dupont, Panzer, and Pym do define virtual morphisms of DF log structures their definition of a DF log structure is different from ours. Thus, we give a somewhat different definition here. 
\begin{definition}
    A virtual morphism of DF log structures on a scheme $X$ 
    $$\tup{s_i\co \mcl O_X \to \mcl L_i}_{1\leq i \leq n} \to \tup{t_j\co \mcl O_X \to \mcl M_j}_{1\leq j \leq m}$$
    is a collection $\{e_{ij}\}$ of integers together with isomorphisms of line bundles
    $$\mcl L_i \xra{\cong} \bigotimes_{1\leq j \leq m}\mcl M_j^{\otimes e_{ij}}$$
    for each $1\leq i \leq n$
    satisfying the following if $s_i \neq 0$
    \begin{itemize}
        \item $e_{ij}\geq 0$ for every $j$.
        \item the sections $s_i$ are mapped to the corresponding sections $\bigotimes_{1\leq j \leq m}t_j^{\otimes e_{ij}}$
    \end{itemize}
\end{definition}
\begin{remark}
    Although the constructions are closely related, our category of DF log schemes and the category of log schemes in \cite{dupont2024logarithmicmorphismstangentialbasepoints} are not equivalent categories nor are they subcategories of each other and thus this definition cannot be ``derived'' from that of Dupont, Panzer, and Pym in any meaningful way. 
\end{remark}
\begin{remark}
    For non connected schemes this definition must be appropriately modified for sections that are identically zero on some, but not all, components. 
\end{remark}
The category of DF log schemes with virtual morphisms, $\vdflogcat$ fits into a commutative diagram of categories:
\begin{center}
    \begin{tikzcd}
        \dflogcat_X\arrow[r] \arrow[d]& \logcat_X\arrow[d]\\
        \vdflogcat_X \arrow[r]& \vlogcat_X
    \end{tikzcd}
\end{center}
We will omit the construction of the functor $\vdflogcat_X\to\vlogcat_X$ but the interested reader is encouraged to construct it themselves and/or verify that our definition of the Kato--Nakayama analytification functor is still well defined for virtual morphisms. Doing this hopefully sheds some light on why this definition is a reasonable one.
\begin{example}\label{example: virt morph point into circle}
    There is no morphism of log schemes $\spec \k \to (\spec \k , 0\co \mcl O_{\k} \to \mcl O_{\k})$ but there is such a morphism of virtual log schemes (one for each automorphism $\mcl O_{\k} \to \mcl O_{\k}$). For $\k = \C$, the Kato--Nakayama analytification of these maps are the inclusions of a point in $\S^1$. Similarly, the group inverse map $i\co\S^1 \to \S^1$ is not the analytification of any map of log schemes $(\spec \C , 0\co \mcl O_\C \to \mcl O_\C) \to (\spec \k , 0\co \mcl O_\C \to \mcl O_\C)$ but it is a virtual map of log schemes corresponding to the isomorphism 
    $\mcl O_\C \cong \mcl O^{\otimes (-1)}_\C$.
\end{example}
\begin{example}
    Let $X$ be a smooth scheme and let $D$ be a smooth effective Cartier divisor. By definition
    $$(X, s_D\co \mcl O_X \to \mcl O_X(D))^\kn = \realbu{D(\C)} X(\C).$$
    By our definition of the real oriented blow-up there is an embedding of $\realbu{D(\C)} X(\C)$, into the unit circle bundle of (the analytification of) $\mcl O_X(D)$ with some arbitrarily chosen metric. This unit circle bundle is, again by definition, the KN analytification of $(X, 0\co \mcl O_X \to \mcl O_X(D))$. The inclusion
    $$(X, s_D\co \mcl O_X \to \mcl O_X(D))^\kn \into (X, 0\co \mcl O_X \to \mcl O_X(D))^\kn$$
    is not the analytification of any map of log schemes since the zero section cannot be pulled back to a non zero section. It is however the analytification of a virtual map given by the identity on both underlying spaces and $\mcl O_X(D) \to \mcl O_X(D)^{\otimes 1}$. 
\end{example}

\section{Topological Fulton--MacPherson and Kontsevich Spaces}\label{sec:topological constructions}
In this section we recall the topological Fulton--MacPherson spaces, $\topfm_n(\R^D)$, and the closely related Kontsevich spaces, $\kont_{D,n} \cong \topfm_n(\R^D)\big /\R_{>0}\ltimes \R^D$. 

\subsection{The Fulton--MacPherson Compactification of a Manifold}\label{subsec:topo Fulton--MacPherson}
Let $X$ be a smooth manifold, let $n$ be a positive integer and let $\psets{n}$, denote the set of subsets of $[n]$ with at lest $2$ elements. For a set $I\in \psets{n}$ let $\Delta_I$ denote the corresponding diagonal in $X^n$, i.e. $$\Delta_I = \{(x_1,\dots, x_n)\in X^n \vert\  x_i = x_j \ \forall\ i,j \in I\}.$$ By abuse of notation we will also use $\Delta_I$ to denote the (small) diagonal in $X^{I}$.

\begin{definition}
    The \emph{Fulton--MacPherson compactification}, $\topfm_n(X)$ is the closure of the image of the map 
    $$\conf_n(X) \into \prod_{I\in \psets{n}} \realbu{\Delta_I}X^I.$$
    We let $f_I\co \topfm_n(X) \to \realbu{\Delta_I}(X)$ denote the restriction of the corresponding projection to $\topfm_n(X)$ and we let $\rho\co \topfm_n(X) \to X^n$ denote the surjective extension of the embedding $\conf_n(X) \into X^n$ to $\topfm_n(X)$.
\end{definition}
\begin{remark}
    This definition is the one used by Axelrod and Singer in \cite{axelrodsinger1994chernsimons} but many sources use a different definition of the topological Fulton--MacPherson compactification introduced by Sinha \cite{sinha2004manifold}. These can easily be shown to be equivalent. 
\end{remark}

\begin{definition}
    The \emph{Kontsevich space} of $n$ points in dimension $d$, $\kont_{D,n}$, is defined as the fibre
    $\rho^{-1}((0,0,\dots,0)) \sub \topfm_n(\R^D)$.
\end{definition}
\begin{remark}
    It is more common to define $\kont_{D,n}$ as $\topfm_m(\R^D)\big / \R_{>0}\ltimes \R^D$ but one can show that the composition $\rho^{-1}((0,0,\dots,0)) \into \topfm_n(\R^D) \to \topfm_m(\R^D)\big / \R_{>0}\ltimes \R^D$ is an isomorphism and this equivalent definition will be more useful for our purposes.
\end{remark}
The fibre over the origin of $\realbu{\Delta_I}(\R^D)^I$ is $\S^{D(\abs{I}-1)-1}$. Thus, $f_I\co \topfm_n(\R^D) \to \realbu{\Delta_I}(\R^D)^I$ restricts to a map $\pi_I\co \kont_{D,n} \to \S^{D(\abs{I}-1)-1}$. These maps give a closed embedding
$$\kont_{D,n} \into \prod_{I\in \psets{n}}\S^{D(\abs{I}-1)-1},$$
i.e. an element $x\in \kont_{D,n}$ is uniquely determined by its components $x_I = \pi_I(x)$.

Note that while we have chosen to use $[n] = \{1,\dots,n\}$ as coordinate indices in the definitions above we could have used any (finite) index set $N$ instead. We can thus define $\topfm_N(X)$ and $\kont_{d,N}$ with maps $\rho\co \topfm_N(X)\to X^N$, $f_{I}\co \topfm_N(X) \to \realbu{\Delta_I}(X)^I$, and $\pi_{I}\co \kont_{D,N} \to \S^{D(\abs{I}-1)-1}$ in the analogous way. Here $I$ is a subset of $N$.
\subsection{The Fulton--MacPherson Operad}
For a fixed dimension $D$ the collection $\set{\kont_{D,n}}_{n\in \N}$ can be given the structure of a topological operad as follows.

The permutation action of $\Sigma_n$ on $\conf_n(\R^D)$ extends to a free action on $\topfm_n(\R^D)$ which gives an induced action on the fibre over the origin, $\kont_{D,n}$. This is the symmetry action of the operad.

Next, any surjection $\vp \co I\quot J$ of sets induces a monomorphism $(\R^D)^J \into (\R^D)^I$ such that the inverse image of $\Delta_I$ is $\Delta_J$. This gives a map 
$\realbu{\Delta_J}(\R^D)^J\to \realbu{\Delta_I}(\R^D)^I$ which restricts to a map of fibres over the origin 
$$g_\vp\co \S^{D(\abs{J}-1)-1} \to \S^{D(\abs{I}-1)-1}.$$
Let $q\co M\quot [n]$ be a surjection of finite sets.
By abuse of notation we will also let $q$ denote the restriction of this function $I\quot q(I)$ for any $I \sub M$. The above construction allows us to define a map 
$$\gamma\co \kont_{D,[n]} \times \prod_{i\in [n]} \kont_{D,q^{-1}(i)} \to \kont_{d,M}$$
by
$$\gamma(x,y^1,\dots,y^n)_I = \begin{cases}
    y^i_I & I \sub q^{-1}(i)\\
    g_q(x_{q(I)}) & \textnormal{else}
\end{cases}.$$
This defines the composition maps of the operad.
\begin{remark}
    We use the standard notation for an operad defined in terms of finite sets here. See for example \cite{markl2007operads} for an introduction to the notation.
\end{remark}

Finally, $\kont_{d,1}$ is the one point space so there is only one choice for the identity map, $\eta\co * \to \kont_{d,1}$. 
The operad axioms can be verified manually for this collection of maps. 
\begin{definition}
    We call the operad with spaces $\set{\kont_{D,n}}$ (for a fixed $D$) and the above composition, symmetry, and identity maps the \emph{Fulton--MacPherson operad in dimension $D$}, denoted $\tk_D$.
\end{definition}

By a result of Salvatore, the Fulton--MacPherson operad is weakly equivalent to the operad of little $d$-dimensional disks \cite[proposition 3.9]{salvatore1999configuration}. 

The action of $\SO(D)$ on $\R^D$ induces an action on $\conf_n(\R^D)$ which extends to $\topfm_n(\R^D)$. Since all rotations map the zero vector to itself this restricts to an action on the fibre $\kont_{D,n}$. One can check that this gives an action by $\SO(D)$ on the $\tk_D$ operad in the sense of \cite{salvatore2003framed}. Furthermore the weak equivalence $\tk_D \sim \ld_D$ can be shown to commute with this action which gives a weak equivalence between $\tk_D\rtimes \SO(D)$ and $\ld_D \rtimes \SO(D)$, the framed little disks operad.

With the notable exception of the case $D=2$, the framed little disks operad is, however, not the operad studied in this article. Instead we are interested in the following construction. 
When the dimension is an even number $D = 2d$ there is an embedding of topological groups 
$$\S^1 \cong U(1) \xhookrightarrow{\textnormal{diag.}} U(d)\into\SO(2d).$$
This induces an action of $\S^1$ on the $\tk_{2d}$ operad and we can form the semidirect product $\tk_{2d} \rtimes \S^1$. 

The main goal of this article is to construct a non-unital operad of log schemes whose Kato--Nakayama analytification is $\tk_{2d}\rtimes \S^1$ which is weakly equivalent to the $\S^1$-framed little disks operad.

\section{The Chen, Gibney, Krashen Operad}\label{sec: geometric operad}
Fix a base field $\k$ for the remainder of the article.
In this section we recall the definition of the moduli spaces of stable $n$-pointed rooted trees of $d$-dimensional projective spaces, $T_{d,n}$, and describe an operad of (smooth) $\k$-varieties with spaces $T_{d,n}$ for a fixed $d$ which generalizes the operad of pointed stable curves of genus $0$. For a more in depth description of the spaces $T_{d,n}$, including a definition of a rooted tree of projective spaces, see \cite{chen2009pointed}.

\subsection{Fulton--MacPherson Compactifications}\label{subsec: alg fm}
Let us recall some important aspects of the Fulton--MacPherson compactification \cite{fulton1994acompactification}. Let $X$ be a smooth $\k$-variety. For an integer $n$ and a subset $I \sub [n]$ with $\abs{I}\geq 2$ let $\mcl I_I$ denote the sheaf of ideals on $X^n$ corresponding to the $I$-diagonal, $\Delta_I \sub X^n$. Recall that we denote the set of subsets of $[n]$ with at least $2$ elements by $\psets{n}$.

Let $\rho\co \afm_n(X) \to X^n$ denote the Fulton--MacPherson compactification of $\conf_n(X)$. We call this a ``compactification'' of $\conf_n(X)$ since $\rho$ restricted to $\rho^{-1}(\conf_n(X))$ is an isomorphism and $\afm_n(X)$ is compact provided that $X$ is.
The remainder of this section will be dedicated to recalling some important results relating to the Fulton--MacPherson compactification. First, we just state the following facts, all of which can be found with proofs in \cite{fulton1994acompactification}.
\begin{itemize}
    \item The scheme $\afm_n(X)$ is a smooth variety.
    \item There is an ordering of the elements in $\psets{n}$, such that $\rho \co \afm_n(X) \to X^n$ is given by a sequence of blow-ups where we start with $X^n$ and at step $t$ of the sequence we blow-up in the dominant transform of the $t$th diagonal through all previous blow-ups.
    \item The closed subset $D(I) \sub \afm_n(X)$ given by the dominant transform of $\Delta(X)$ through all these blow-ups is a smooth, effective Cartier divisor. 
    \item The divisors $D(I)$ form a strict normal crossings divisor, i.e. any set of the divisors $D(I)$ meets transversely. 
    \item The intersection $D(I_1,\dots, I_k)\defeq D(I_1) \cap \dots \cap D(I_k)$ is non-empty if and only if the sets $I_1, \dots, I_k$ are \emph{nested} i.e. for each $i,j$ we have $I_i \sub I_j$, $I_j \sub I_i$ or $I_i\cap I_j = \emptyset$.
    \item We have $\afm_n(X) \setminus \conf_n(X) = \bigcup_{I \in\psets{n}} D(I)$.
    \item We have that $\rho^{-1}(\Delta(I)) = \bigcup_{I \sub I'} D(I')$, i.e. the pullback of the ideal inclusion $\mcl I_I \into \mcl O_{X^n}$ factors via a quotient
    $$\rho^* \mcl I_I \quot \prod_{I \sub I'} \mcl I_{D(I')} \cong \bigotimes_{I \sub I'} \mcl O_{\afm_n(X)}(-D(I')).$$
    \item The composition $\afm_n(X) \to X^n \to X^I$ factors as $\afm_n(X) \xra{f_I} \bu{\Delta_I} X^I \xra{\rho_I} X^I,$ where $\rho_I$ is the blow-down map.
\end{itemize}
\begin{remark}
    Recall that for a blow-up $\bu{Y}X \to X$ the \emph{dominant} transform of a closed subscheme $Z\sub X$ is defined as the strict transform of $Z$ if $Z\not\sub Y$ and the pullback of $Z$ to $\bu{Y}X$ if $Y\sub Z$.
\end{remark}
Note that by the last two points we have that if $E_{\Delta_I} \sub \bu{\Delta_I} X^I$ denotes the exceptional divisor then there is an isomorphism $$f_I^* \mcl O(E_{\Delta_I}) \cong \bigotimes_{I \sub I'} \mcl O_{\afm_n(X)}(D(I'))$$ 
which sends corresponding sections to each other. 

In addition to this geometric description of $\afm_n(X)$ there is also a functorial description which we will make great use of in this article.
\begin{definition}
    A \emph{screen} on a scheme $H$ is a map $f\co H \to X^n$ together with the following data
    \begin{itemize}
        \item an invertible quotient $f^* \mcl I_{I} \quot \mcl L_I$ for each $I\in \psets{n}$
        \item a map $\mcl L_I \to \mcl L_J$ for each $I \sub J$
    \end{itemize}
    such that the following diagram commutes for every $I \sub J$
    \begin{center}
        \begin{tikzcd}
            f^*\mcl I_{I} \arrow{r}\arrow{d}& \mcl L_{I}\arrow{d}\\
            f^*\mcl I_{J} \arrow{r} & \mcl L_{J}
        \end{tikzcd}
    \end{center}
    where $f^*\mcl I_I \to f^* \mcl I_J$ is the pullback of the corresponding inclusion. 

    We say a screen satisfies the $I$-vanishing property for some $I\in \psets{n}$ if $\textnormal{pr}_i\circ f\co H \to X = \textnormal{pr}_j\circ f\co H \to X$ for each $i,j\in I$ and, for every $J\in \psets{n}$, if $J\not \sub I$ and $\abs{I\cap J} \geq 2$ then $\mcl L_{I\cap J}\to \mcl L_J$ is the zero morphism.
\end{definition}
We can define a screen on $\afm_n(X)$ by taking $f=\rho$, $\mcl L_I = \bigotimes_{I \sub I'}\mcl O_{\afm_n(X)}(-D(I'))$ and letting the invertible quotients be the maps
$$\rho^*\mcl I_I \quot \bigotimes_{I \sub I'}\mcl O_{\afm_n(X)}(-D(I')).$$
The maps $\mcl L_{I} \to \mcl L_J$ here are induced by the ideal inclusions $\mcl O_{\afm_n(X)}(-D(I))\into \mcl O_{\afm_n(X)}$. We call this screen the \emph{universal screen} on $\afm_n(X)$.
\begin{definition}
    Let $\afmfun_n\co \schcat^{op} \to  \setcat$ be the functor which sends a scheme $H$ to the set of screens on $H$ up to isomorphism of the screen data. 

    Similarly let $\afmfun(I_1,\dots,I_k)$ denote the subfunctor sending $H$ to the set of screens on $H$ up to isomorphism which satisfy the $I_r$-vanishing property for each $1\leq r \leq k$.
\end{definition}
\begin{proposition}\label{prop: FM functor}
    The Fulton--MacPherson compactification $\afm_n(X)$ represents the functor $\afmfun_n$. Furthermore, the natural isomorphism $\hom(-,\afm_n(X)) \to \afmfun_n$ is given by sending the element $f\co H \to \afm_n(X)$ to the pullback to $H$ of the universal screen on $\afm_n(X)$ described above.

    Similarly $D(I_1,\dots, I_k) \sub \afm_n(X)$ represents $\afmfun(I_1,\dots, I_k)$. The natural isomorphism is given in the same way here as above.
\end{proposition}
\begin{proof}
    This is theorem 4 in \cite{fulton1994acompactification}.
\end{proof}
\subsection{Pointed Rooted Trees of Projective Spaces}\label{subsec: pointed trees}
Chen, Gibney, and Krashen \cite{chen2009pointed} introduced \emph{stable pointed rooted trees of projective spaces} as a higher-dimensional analogue of stable pointed rational curves, and showed that $T_{d,n}$ (defined in \cref{subsec: Tdn def and properties}) is the moduli space of stable rooted $n$-pointed trees of 
$d$-dimensional projective spaces.
Here we recall the definition and describe the morphisms in the operad $\gk_d$ (defined in \cref{subsec: Tdn def and properties}), with $\gk_d(n) =T_{d,n}$ in terms of natural operations on such trees. The discussion is intended to convey geometric intuition for \cref{subsec: Tdn def and properties}, and we therefore suppress most proofs and technical details.

\begin{definition}
A \emph{rooted tree of $d$-dimensional projective spaces} is a connected, reduced, projective scheme $T$ with a closed embedding $r_0\co \P^{d-1} \into T$, called the \emph{root hyperplane}, obtained by the following finite inductive process:
\begin{enumerate}
    \item  
    Start with $T \cong \mathbb{P}^d$ and choose a hyperplane embedding
    $$r_0 \co \mathbb{P}^{d-1} \hookrightarrow V.$$ 
    This embedding is the \emph{root hyperplane}.
    
    \item
    Given a tree of $d$-dimensional projective spaces $T$ produced in a previous step, choose a smooth point $x \in T$ disjoint from the singular locus and the root hyperplane $r_0\co \P^{d-1} \into T$.  
    The next iteration $T'$ is defined by forming the blow-up $\bu{x} T$ and gluing along its exceptional divisor $\mathbb{P}^{d-1}$ a new copy of $\mathbb{P}^d$ along some hyperplane
    $\mathbb{P}^{d-1} \sub \mathbb{P}^d.$ Since $x$ was chosen to be disjoint from the root hyperplane $r_0$ induces an embedding $r_0'\co \P^{d-1} \into T'$, which we take to be the new root hyperplane.
    Repeat finitely many times.
\end{enumerate}

An \emph{$n$-pointed rooted tree} is such a rooted tree of projective spaces $T$ together with $n$ distinct  points
$$p_1, \dots, p_n : \spec k \hookrightarrow V$$
lying outside the singular locus $\sing(T)$ and outside the root hyperplane.
\end{definition}
\begin{remark}
    When $d=1$, blowing up in a smooth point does nothing and hyperplanes are points so an $n$-pointed rooted tree of $1$-dimensional projective spaces is just an $n+1$-pointed genus $0$ curve with at worst nodal singularities.
\end{remark}

Decompose $T=\bigcup_i T_i$ into irreducible components and let $T_0$ denote the \emph{root component}, i.e. the component containing the root hyperplane.
Each $T_i$ is isomorphic to $\mathbb{P}^d$ blown up at a closed subscheme $Q_i$ consisting of finitely many disjoint $\k$-points, hence admits a canonical
\emph{blow-down map}
$$b_i: T_i \longrightarrow \mathbb{P}^d.$$
Furthermore, there is an embedding $r_i\co \P^{d-1}\into T_i$ for each $T_i$ such that $b_i\circ r_i\co \P^{d-1} \to \P^d$ is the inclusion of a hyperplane. For the root component $r_0$ is the root hyperplane and for other components $r_i$ is the embedding of the hyperplane along which the component was attached.
Note that the disjoint union of $r_i$ and the exceptional divisor $E_i$ of the blow-up $T_i \to \P^d$ is precisely the intersection $T_i\cap \sing(T)$. 

\begin{definition}
With notation as above, the $n$-pointed rooted tree $T$ is \emph{stable} if, for every component $T_i$,
there are at least two corresponding \emph{special points}, meaning the marked points $\{p_1,\dots, p_n\}\cap T_i$ and the blown-up points $Q_i$.
\end{definition}
\begin{remark}
    This condition is equivalent to $T$ having no non-trivial automorphisms restricting to the identity on the marked points and the root hyperplane. 
\end{remark}
\begin{remark}
When $d=1$, the above condition is equivalent to there being three special points on each component if we think of the corresponding hyperplanes as special points as well. This is the usual stability condition for pointed genus $0$ curves with at worst nodal singularities.
\end{remark}

By Theorem 3.4.4. in \cite{chen2009pointed} the smooth varieties $T_{d,n}$ are moduli spaces of stable $n$-pointed rooted trees of $d$-dimensional projective spaces. There are no stable $1$-pointed trees but in this case we define $T_{d,1}= \spec \k$. As mentioned before, a hyperplane in $\P^1$ is just a point and thus we have $T_{1,n} \cong \overline{M}_{0,n+1}$, the moduli space of stable $n+1$-pointed rational curves where the additional point is the root hyperplane.

Each variety $T_{d,n}$ comes with a canonical action by the symmetric group $\Sigma_n$ corresponding to permutation of the marked points. Furthermore, for each $1\leq i \leq n$ there is a morphism
$$T_{d,n}\times T_{d,m}\xra{\circ_i} T_{d,n+m-1}$$
which takes an $n$-pointed tree $T$ and an $m$-pointed tree $T'$ and maps them to the $n+m-1$ pointed tree formed by blowing-up $T$ in the $i$th marked point and gluing the resulting exceptional divisor to the root hyperplane in $T'$. It is clear from definitions that these maps give rise to an operad of varieties, which we denote $\gk_d$, with spaces $\gk_d(n) = T_{d,n}$. In the case $d=1$ this construction is classical, see for example \cite{getzler98modular}. 

\subsection{Definition and Properties of $T_{d,n}$}\label{subsec: Tdn def and properties}
In this section we recall the definition and some important properties of the $T_{d,n}$ spaces from \cite{chen2009pointed}.
We also introduce a functor of points for these spaces and use this to define the operad $\gk_d$ with spaces $\gk_d(n) = T_{d,n}$ which was informally described in \cref{subsec: pointed trees}. We will not explicitly prove that $\gk_d$ is isomorphic to the operad described in \cref{subsec: pointed trees} but it is not hard to show using Theorem 3.4.4. in \cite{chen2009pointed}.

If $\textnormal{pr}_i\co X^n \to X$ denotes the projection to the $i$th component it is clear from definitions that the composition 
$$D([n])\into \afm_n(X) \xra{\rho} X^n \xra{\textnormal{pr}_i} X$$
is independent of the integer $i$. We denote this morphism by $\pi \co D([n]) \to X$.
In \cite{chen2009pointed}, Chen, Gibney, and Krashen show that for any $\k$-valued point $x\in X$ the fibre $\pi^{-1}(x)$, depends only on $n$ and $\dim X$ up to (non-canonical) isomorphism provided that $X$ is smooth. We define $T_{d,n}$ as $\pi^{-1}(0)$ for the map $\pi\co D([n]) \to \A^d$ but by the above we could equally well have taken any other smooth variety instead of $X= \A^d$.
\begin{remark}
    Just like the topological case notice that while we have up until this point indexed coordinates with $[n]$ we are free to use any finite index set $I$ and $\afm_I(X)$ and $T_{d,I}$ together with all related maps and divisors are defined in the analogous way. 
\end{remark}

The variety $T_{d,n}$ is smooth and comes equipped with a smooth closed subscheme $T_{d,n}(I)$ for every non empty $I \sub [n]$. For $I \in\psets{n}$, i.e. $I\neq [n]$ and $\abs{I}\geq 2$, we define $T_{d,n}(I)$ as the pullback of
$D(I)\sub \afm_n(X)$ to $T_{d,n}$. For other $I$ (i.e. $I$ is $[n]$ or a one point set) we take $T_{d,n}(I) \defeq T_{d,n}$. Similarly, for a collection of (distinct) subsets $I_1,\dots, I_k \sub [n]$ we define $T_{d,n}(I_1,\dots, I_k) \defeq \bigcap_{i}T_{d,n}(I_i)$. For $I\neq [n]$ and $\abs{I}\geq 2$ the subscheme $T_{d,n}(I)$ is a smooth divisor of $T_{d,n}$ and the union of all these divisors is a normal crossings divisor. Additionally, the intersection $T_{d,n}(I_1,\dots, I_k)$ is non empty if and only if $(I_1,\dots, I_k)$ is \emph{nested}. All this essentially follows from the analogous results for $\afm_n(X)$ recalled in \cref{subsec: alg fm}. Details can be found in \cite{chen2009pointed}.
\begin{remark}
    In terms of pointed trees $T_{d,n}(I)$, where $I\in \psets{[n]}$, parameterizes trees which have two (not necessarily irreducible) components such that the component which does \emph{not} contain the root hyperplane contains exactly those marked points indexed by $I$.
\end{remark}

The first goal of this section will be to give a functor of points description of the schemes $T_{d,n}(I_1,\dots, I_k)$. For this we introduce the following. 
\begin{definition}\label{def: diff sheaf}
    For a scheme $H$, an integer $d$, and a finite, non-empty, set $I$ with at least two element we let $\mcl F^{H,d}_I$ denote the quasi-coherent sheaf of modules defined by
    $$\mcl F^{H,d}_I \defeq \left\langle \{t_{ij}^k\}_{i,j \in I}^{1\leq k\leq d}\right \rangle \Bigm / \sim.$$
    Here $\left\langle \{t_{ij}^k\}_{i,j \in I}^{1\leq k\leq d}\right \rangle$ denotes the free module with generators $t_{ij}^k$ and $\sim$ is generated by sections of the form $t_{ij}^k +t_{jl}^k-t_{il}^k$. We will also denote this by $\mcl F_I^d$ or $\mcl F_I$ if $H$ or $H,d$ are clear from context. 
    
    For every function of sets $\vp\co I\to  J$ we let $F_{\vp}\co \mcl F_{I}^{H,d} \into \mcl F_J^{H,d}$ denote the sheaf morphism defined by sending
    $t_{ij}^k \mapsto t_{\vp(i)\vp(j)}^k.$
\end{definition}
\begin{remark}
    Note that $\mcl F^{H,d}_I$ is free of rank $d(\abs{I}-1)$. Furthermore, note that the sheaves $\mcl F^{H,d}_I$ and the maps $F_{\vp}$ are pullbacks of corresponding vector bundles with morphisms on $\spec \k$. For all $I \in\psets{n}$ these are in turn isomorphic to the pullbacks of the ideal sheaves $\mcl I_{\Delta_I}$ on $(\A^d)^n \cong \spec \k[x^k_i]_{1\leq i \leq n}^{ 1\leq k \leq d}$ via the inclusion of the origin $\spec \k \into (\A^d)^n$. Here we identify the section $t_{ij}^k$ of $\mcl F^{d}_I$ with the pullback of the global section $x_i^k-x^k_j \in I_{\Delta_I} \sub \k[x^k_i]$. Additionally, if $I\sub J \sub [n]$ the map $F_{I\into J}$ of sheaves on $\spec \k$ is the pullback of the ideal inclusion map $\mcl I_{\Delta_I} \into \mcl I_{\Delta_J}$. 
\end{remark}
\begin{definition}
    Let $\tfun{d,n}\co Sch^{op}\to Set$ be the functor which sends $H$ to the set of collections of invertible quotients, 
    $$\set{\phi_I\co \mcl F^{H,d}_I \quot \mcl L_I}_{I \in \psets{n}},$$
    such that for every $I \sub J$ there is a morphism $\mcl L_I \to \mcl L_J$ making the following diagram commute.
    \begin{center}
        \begin{tikzcd}
            \mcl F_{I}^{H,d} \arrow{r} \arrow[d,"F_{I\sub J}"] & \mcl L_{I} \arrow{d} \\
            \mcl F_{J}^{H,d} \arrow{r}& \mcl L_{J}
        \end{tikzcd}
    \end{center}
    Such a collection will sometimes be referred to as a \emph{simple screen}.
    Similarly let $\tfun{d,n}(I_1,\dots, I_k)$ denote the subfunctor of $\tfun{d,n}$ with the added condition that for every $I_r$, the simple screens satisfy the \emph{$I_r$-vanishing property}, i.e. if $J\nsubseteq I_r$ and $\vert J\cap I_r\vert \geq 2$ then $\mcl L_{J\cap I_r} \to \mcl L_J$ is the zero morphism. 
\end{definition}
\begin{remark}
    Note that $\tfun{d,n}([n]) = \tfun{d,n}(\set{i}) = \tfun{d,n}$.
\end{remark}
There is a natural simple screen on $T_{d,n}$ which can be constructed as follows. Let $i\co T_{d,n}\into \afm_n(\A^d)$ denote the inclusion of the fibre over the origin in $(\A^d)^n$. The ideal sheaves $\mcl I_I$ pull back to $i^* \mcl I_I \cong \mcl F_I^d$ and for $I\sub J$ the inclusion $\mcl I_I \into \mcl I_J$ pulls back to the map $F_{I\sub J}\co \mcl F_I^d\to \mcl F_J^d$. Thus, the universal screen on $\afm_n(\A^d)$ pulls back to a simple screen on $T_{d,n}$ via $i$. We will refer to this as the \emph{universal simple screen} on $T_{d,n}$. Let $\mcl F_I\quot \mcl M_I^{d,n}$ denote the invertible quotient corresponding to $I \in\psets{n}$ of the universal simple screen on $T_{d,n}$. We will also denote this by $\mcl M_I$ if $d,n$ are clear from context.
\begin{lemma}\label{lemma: M_I is tensor prod of O(I') duals}
    Let $i\co T_{d,n} \to \afm_n(\A^d)$ denote the inclusion of the fibre over the origin. There are unique isomorphisms
    $$\bigotimes_{I \sub I'} i^* \mcl O(-D(I')) \xra{\cong} \mcl M_I$$
    such that for $I \sub R$ the following diagram commutes
    \begin{center}
        \begin{tikzcd}
            \bigotimes_{R\sub R'} i^* \mcl O(-D(R')) \arrow[r, "\cong"]\arrow[d,"\bigotimes_{S\not \sub R'} i^*t_{R'}" ] & \mcl M_R\arrow[d]\\
            \bigotimes_{I \sub I'} i^* \mcl O(-D(I))\arrow[r, "\cong"] & \mcl M_I
        \end{tikzcd}
    \end{center}
    where $t_{R'}\co O(-D(R')) \into \mcl O$ denotes the ideal inclusion.
\end{lemma}
\begin{proof}
    This follows from the analogous result for the universal screen on $\afm_n(X)$.
\end{proof}

\begin{proposition}\label{prop: T_{d,n} functor}
    The functor $\tfun{d,n}$ is represented by $T_{d,n}$ and similarly $\tfun{d,n}(I_1,\dots, I_k)$ is represented by $T_{d,n}(I_1,\dots, I_k)$. Furthermore, the natural isomorphism $\hom(-,T_{d,n}) \to \tfun{d,n}$ is given by sending a map $f\co H \to T_{d,n}$ to the pullback of the universal simple screen on $T_{d,n}$ to $H$. The analogous result holds for $\hom(-,T_{d,n}(I_1,\dots,I_k)) \to \tfun{d,n}(I_1,\dots,I_k)$.
\end{proposition}
\begin{proof}
    First note that if $i_0\co \spec \k \to (\A^d)^n$ denotes the inclusion of the origin then we have seen that $i_0^*\mcl I_I = \mcl F_I$ for every $I \in\psets{n}$, where $\mcl I_I$ is the ideal sheaf corresponding to the $I$-diagonal in $(\A^d)^n$.
    Since 
    $$T_{d,n} = \pi^{-1}(0) = D([n])\times_{\A^d} \spec \k \cong D([n])\times_{(\A^d)^n} \spec \k$$ 
    the result follows immediately from \cref{prop: FM functor}. The same argument also gives the $T_{d,n}(I_1,\dots, I_k)$ case.
\end{proof}

This functorial description simplifies the proof of many of the following results greatly.
Let $n$ be a positive integer, and let $q\co [n]\quot M$ be a surjection of sets. We also let $q$ to denote the restriction of this map $J \to q(J)$ for any $J\sub M$.

\begin{proposition}\label{prop:comp-morph-iso-onto-closed-subscheme}
    The composition map described in \cref{subsec: pointed trees} is well defined and gives an isomorphism
    $$T_{d,n}\times \prod_{r=1}^n T_{d,q^{-1}(r)} \cong T_{d,M}(q^{-1}(1),\dots, q^{-1}(n)).$$
\end{proposition}
\begin{proof}
    We will construct an explicit natural isomorphism between the functors represented by the left and the right side. Identifying the induced isomorphism of schemes with the map described in \cref{subsec: pointed trees} is left as an unimportant exercise for the reader.
    
    We define this natural transformation 
    $$\tfun{d,n}\times \prod_{r=1}^n \tfun{d,q^{-1}(r)}\to \tfun{d,M}(q^{-1}(1),\dots,q^{-1}(n))$$
    by sending simple screens (on some $\k$-scheme $H$) 
    $$\set{\phi_I\co \mcl F_{I}\to \mcl L^0_I}_{I\in \psets{n}} \times \prod_{r=1}^{n}\set{\psi^r_{J}\co \mcl F_{J}\to \mcl L^r_{J}}_{J\in \psets{q^{-1}(r)}} \to \set{\rho_K\co \mcl F_K\to \mcl L_K}_{K\in \psets{M}},$$
    where $\set{\rho_K\co \mcl F_K\to \mcl L_K}_{K\in \psets{M}}$ is defined as follows.
    We define the line bundles $\mcl L_K$ as 
    $$\mcl L_K = \begin{cases}
        \mcl L^r_{K} & K \sub q^{-1}(r) \textnormal{ some } r\\
        \mcl L^0_{q(K)}& \textnormal{else}
    \end{cases}$$
    and the corresponding quotients as
    $$\rho_K = \begin{cases}
        \psi^r_{K}  & K \sub q^{-1}(r), \textnormal{ some } r\\
        \phi_{q(K)} \circ F_{q}  & \textnormal{else}
    \end{cases}.$$
    For the definition of $F_q\co \mcl F_{K}\to \mcl F_{q(K)}$, see \cref{def: diff sheaf}.
    We define the maps $\mcl L_{K_1} \to \mcl L_{K_2}$, for $K_1 \sub K_2$, to be
    $$\begin{cases}
        \mcl L^r_{K_1} \to \mcl L^r_{K_2} & K_2 \sub q^{-1}(r) \textnormal{ some } r\\
        \mcl L^0_{q(K_1)} \to \mcl L^0_{q (K_2)} & K_1 \not\sub q^{-1}(r) \textnormal{ any } r\\
        0 & K_1 \sub q^{-1}(r) \textnormal{ and } K_2 \not \sub q^{-1}(r)
    \end{cases}.$$
    It is easy to check that $\set{\rho_K}_{K\in \psets{M}}$ with the maps $\mcl L_{K_1} \to \mcl L_{K_2}$ above is a simple screen with the $q^{-1}(r)$-vanishing property for every $r$ and so this transformation is well defined on the level of sets. Furthermore, the map
    is clearly natural and so we have defined a natural transformation as desired. 
    
    To show that this gives a natural isomorphism first note that that this map is clearly injective. Surjectivity follows from the fact that if $\set{\rho_K\co \mcl F_K\to \mcl L_K}_{K\in \psets{M}}$ is a simple screen satisfying the $q^{-1}(r)$-vanishing property for every $r$ then, for every $K$ not contained in any $q^{-1}(r)$ the map $\rho_K\co \mcl F_K\to \mcl L_K$ factorizes uniquely as 
    \begin{center}
    \begin{tikzcd}
        \mcl F_{K} \arrow{rd}{\rho_{K}} \arrow{r}{F_q } & \mcl F_{q(K)} \arrow{d}\\
        & \mcl L_{K}
    \end{tikzcd}
    \end{center}
    by the universal property of the quotient. 
\end{proof}
\begin{remark}
    Since we defined $T_{d,n}(\set{i}) = T_{d,n}([n]) = T_{d,n}$ this statement is true even in the case $n= 1$ or $\abs{q^{-1}(r)} =1$.
\end{remark}

\begin{corollary}\label{cor:composition maps restricted to closed subschemes}
    For any $I \in \psets{n}$, the isomorphism of the proposition restricts to an isomorphism of closed subvarieties
    $$T_{d,n}(I)\times \prod_{r=1}^nT_{d,q^{-1}(r)}\cong T_{d,M}(q^{-1}(1), \dots, q^{-1}(n), q^{-1}(I)).$$
    
    Similarly, for any $I_r\in \psets{q^{-1}(r)}$, the isomorphism of the proposition restricts to an isomorphism of closed subschemes
    $$T_{d,n}\times T_{d,q^{-1}(1)}\times\dots\times T_{d,q^{-1}(r)}(I_r)\times \dots\times T_{d,q^{-1}(n)} \cong T_{d,M}(q^{-1}(1), \dots, q^{-1}(n), I_r).$$
\end{corollary}
\begin{proof}
    The proof is analogous.
\end{proof}
\begin{corollary}\label{lemma: pullback of universal screen via composition map}
    Let $\gamma\co T_{d,n} \times \prod_{i=1}^n T_{d,q^{-1}(i)} \to T_{d,M}$ be the composition of the isomorphism above with the inclusion into $T_{d,M}$ and let $\pi_0\co T_{d,n} \times \prod_{i=1}^n T_{d,q^{-1}(i)} \to T_{d,n}$ and $\pi_r \co T_{d,n} \times \prod_{i=1}^n T_{d,q^{-1}(i)} \to T_{d,q^{-1}(r)}$ denote the corresponding projections. There are unique isomorphisms relating the elements of the corresponding universal screens
    $$\gamma^*\mcl M_I^{d,M} \xra{\cong} \begin{cases}
        \pi_r^* \mcl M^{d,q^{-1}(r)}_{I} & I \sub q^{-1}(r) \textnormal{ some } r\\
        \pi^*_0\mcl M^{d,n}_{q(I)}& \textnormal{else}
    \end{cases}.$$
    making (one of) the following diagrams commute:
    \begin{center}
        \begin{tikzcd}
            \mcl F_I\arrow[d, two heads] \arrow[r,two heads, "\textnormal{id}"]& \mcl F_{I} \arrow[d, two heads]\\
            \gamma^*\mcl M_I^{d,M} \arrow[r,"\cong"] & \pi^*_r \mcl M^{d,q^{-1}(r)}_{I}            
        \end{tikzcd}
        \begin{tikzcd}
            \mcl F_I\arrow[d, two heads] \arrow[r,two heads, "F_{q}"]& \mcl F_{q(I)} \arrow[d, two heads]\\
            \gamma^*\mcl M_I^{d,M} \arrow[r,"\cong"] & \pi^*_0 \mcl M^{d,n}_{q(I)}            
        \end{tikzcd}
    \end{center}
\end{corollary}
\begin{proof}
    This follows immediately from the construction and the definition of the universal screen on $T_{d,n}$.
\end{proof}

The next relevant property of the $T_{d,n}$ spaces is that the $\Sigma_n$ action informally described in \cref{subsec: pointed trees} is well defined. Indeed, this action is precisely the restriction to $T_{d,n}$ of the $\Sigma_n$ action on the Fulton--MacPherson compactification $\afm_n(\A^d)$. For each $\sigma \in \Sigma_n$, it is easy to verify that this action corresponds to the action on $\tfun{d,n}$ given by
$$\set{\phi_I\co \mcl F_I \quot \mcl L_I}_{I \in\psets{n}} \in \tfun{d,n}(H) \mapsto \set{\phi_{\sigma^{-1}(I)} \circ\sigma^{-1}_I \co \mcl F_I \quot \mcl L_{\sigma^{-1}(I)}}_{I \in\psets{n}} \in \tfun{d,n}(H),$$
where
$$\sigma_I \co \mcl F_I \to \mcl F_{\sigma(I)}$$
is defined by
sending $t_{ij}^k \mapsto t^k_{\sigma(i)\sigma(j)}$. 

With the maps we have constructed so far we can define the operad of schemes $\gk_d$ which was informally described in \cref{subsec: pointed trees}. Its objects are $\gk_d(n) = T_{d,n}$, its composition maps are the morphisms of \cref{prop:comp-morph-iso-onto-closed-subscheme} (composed with the inclusion $T_{d,M}(q^{-1}(1),\dots, q^{-1}(n)) \into T_{d,M}$), and its symmetry action is the one described above. There is only one candidate for the unit morphism of the operad since $T_{d,1}\cong \spec \k$. It is easy to, at least intuitively, see why these maps satisfy the operad axioms using their $n$-pointed rooted tree descriptions from section \cref{subsec: pointed trees}. Verifying that the axioms hold using the functorial descriptions of all relevant maps here is a more tedious, but easy, exercise. 
\begin{definition}
    The \emph{Chen, Gibney, Crashen operad}, $\gk_d$ is the operad with objects $\gk_d(n) = T_{d,n}$ and with morphisms as described above.
\end{definition}
\begin{remark}
    This is an operad of smooth $\k$-varieties where all composition morphisms are closed embeddings and the unit map is the identity.
\end{remark}

The case $d=1$ here is of special interest. The operad $\gk_1$ is isomorphic to the ``Deligne--Mumford''-operad whose objects are the moduli spaces of stable $n+1$-pointed curves of genus $0$, $\overline{\mcl M}_{0,n+1}$. Indeed, as mentioned earlier, Chen, Gibney, and Krashen show that there are isomorphisms $T_{1,n} \cong \overline{\mcl M}_{0,n+1}$ and one can verify that these isomorphisms can be chosen such that they commute with the respective operad morphisms. 

\section{Log Geometric Constructions} \label{sec: log geometric operad}
In this section we define log structures on the smooth varieties encountered in \cref{sec: geometric operad} and extend the various morphisms to morphisms of log schemes. To this end the following lemma will be used repeatedly. 
\begin{lemma}\label{lemma: divisor iso from scheme theoretic inverse image}
    Let $f\co X\to Y$ be a map of schemes and let $D\sub X$, $E\sub Y$ be effective Cartier divisors such that $f^{-1}(E) =D$ (scheme theoretically). Then there is a unique isomorphism
    $$f^*\mcl O_Y(E) \xra{\cong} \mcl O_X(D)$$
    taking $f^*s_E \mapsto s_D$.
\end{lemma}
\begin{proof}
    First note that since $f^{-1}(E) = D$ there is an epimorphism $f^*\mcl I_E \quot \mcl I_D$ commuting with the corresponding maps to $\mcl O_X$. Since $\mcl I_E$ and $ \mcl I_D$ are ideal sheaves of divisors they are line bundles and since any quotient map of line bundles is an isomorphism this is an isomorphism. Taking duals now yields the result. Uniqueness is clear since an isomorphism of line bundles is uniquely determined by its value for a single generically non-zero section.
\end{proof}

\subsection{The Log Schemes $\lfm_n(X)$, $\msf K_{d,n}$ and $\msf T_{d,n}$}\label{subsec: log fm}
First, let us define the log geometric Fulton--MacPherson compactification which has the nice property that its analytification is the topological Fulton--MacPherson compactification of $\conf_n(\C^d)$, $\topfm_n(\C^d)$.
\begin{definition}
    For a smooth variety $X$, the \emph{log geometric Fulton--MacPherson compactification} (of $\conf_n(X)$), denoted $\lfm_n(X)$, is the log variety with underlying scheme $\afm_n(X)$ together with the line bundles with sections corresponding to the divisors $D(I)$ for every $I \in\psets{n}$.
\end{definition}
Note that all of the line bundles with sections $\mcl O_{\afm_n(X)}(D(I))$ pull back to the trivial line bundle with a nowhere-vanishing section on $\conf_n(X)$. This induces a morphism of log schemes $\conf_n(X) \to \lfm_n(X)$. Furthermore, the following diagram commutes:
\begin{center}
    \begin{tikzcd}
        \conf_n(X) \arrow{rr}\arrow{rrd}\arrow{rrdd} & & \lfm_n(X) \arrow{d}\\
        & & \afm_n(X) \arrow{d}\\
        & & X^n\\
    \end{tikzcd}
\end{center}

Similarly, we can define a log geometric analog of $T_{d,n}$, denoted $\msf K_{d,n}$, by taking the underlying scheme to be $T_{d,n}$ and the log structure given by the pullback of this log structure on $\afm_n(\A^d)$ via the inclusion $T_{d,n} \into \afm_n(\A^d)$. Note that for each $I \in \psets{n}$ with $I \neq [n]$ the line bundle with section corresponding to the divisor $D(I)$ pulls back to the line bundle with section corresponding to $T_{d,n}(I)$ and for $I = [n]$ the pullback of the corresponding line bundle has the zero section. We will let $s_I\co \mcl O_{T_{d,n}}\to \mcl O_{T_{d,n}}(I)$ denote the pullback of $s_I\co \mcl O_{\afm_n(\A^d)} \to \mcl O_{\afm_n(\A^d)}(D(I))$ for each $I \in\psets{n}$.

The operad composition morphisms of the $\gk_d$ operad do not extend to maps of log-schemes for the objects $\msf K_{d,n}$. Instead, in order to extend $\gk_d$ to a (pseudo)-operad of log-schemes we must let its objects be the log-schemes $\msf T_{d,n}$, whose log structures are given by one line bundle with section $s_I\co \mcl O_{T_{d,n}}\to \mcl O_{T_{d,n}}(I)$ for each \emph{non-empty} $I \sub [n]$ defined as above for $\abs{I}\geq 2$ and defined by
$$\mcl O_{T_{d,n}}(\{i\}) \defeq \bigotimes_{I \ni i} \mcl O_{T_{d,n}}(I)^\vee,$$
with $s_{\{i\}} = 0$ for one element sets. For $n=1$ we define $\mcl O_{T_{d,1}}(\{1\}) \defeq \mcl O_{T_{d,1}}$ with the zero section.

\subsection{The Log Geometric CGK Operad}\label{subsec: log geometric operad}
As mentioned, the symmetry and composition morphisms of the $\gk_d$ operad can be extended to morphisms of log schemes $\msf T_{d,n}$. To see this for the symmetry morphisms note that, for $\sigma \in \Sigma_n$ we have that the corresponding automorphism on $\afm_n(X)$, sends each divisor $D(I)$ isomorphically to $D(\sigma(I))$. This induces isomorphisms of line bundles (with sections) $\sigma^*\mcl O_{\afm_n(X)}(D(I)) \to \mcl O_{\afm_n(X)}(D(\sigma^{-1}(I)))$. 
These maps extend the symmetry action on $\afm_n(X)$ to a symmetry action on $\lfm_n(X)$. This in turn gives symmetry actions on $\msf K_{d,n}$ and $\msf T_{d,n}$ in the same way. 

Extending the composition morphisms is slightly trickier. Again, let $n$ be a positive integer and let $q\co M\quot [n]$ be a surjection of sets. As before, let 
$\pi_0\co T_{d,n}\times \prod_{i=1}^n T_{d,q^{-1}(i)} \to T_{d,n}$
and
$\pi_r \co T_{d,n}\times \prod_{i=1}^n T_{d,q^{-1}(i)} \to T_{d,q^{-1}(r)}$
denote the corresponding projections, and let
$$\gamma \co T_{d,n}\times \prod_{i=1}^n T_{d,q^{-1}(i)} \to T_{d,M}$$
be the composition morphism of $\gk_d$. We can immediately show the following.
\begin{lemma}\label{lemma: pullbacks of LB via composition maps easy cases}
    Let $I \subsetneq M$ with $\abs{I} \geq 2$. Then there is a unique isomorphism of line bundles with sections
    \begin{enumerate}[label = (\alph*)]
        \item $\gamma^*\mcl O_{T_{d,M}}(I) \xra{\cong} \pi_r^*\mcl O_{T_{d,q^{-1}(r)}}(I)$ if $I \subsetneq q^{-1}(r)$.
        \item $\gamma^*\mcl O_{T_{d,M}}(I) \xra{\cong} \pi_0^*\mcl O_{T_{d,n}}(q(I))$ if $I =  \bigsqcup_{i\in q(I)} q^{-1}(r)$ and $q(I)\geq 2$.
        \item $\gamma^*\mcl O_{T_{d,M}}(I) \xra{\cong} \mcl O_{T_{d,n}\times \prod T_{d,q^{-1}(i)}}$, where the right hand side is taken with the unit section, if $(I,M_1,\dots, M_n)$ are not nested, i.e. $I$ is not contained in any $q^{-1}(r)$ and also not equal to a union of any collection of $q^{-1}(r)$.
    \end{enumerate}
\end{lemma}
\begin{proof}
    Cases (a) and (b) follow directly from  \cref{cor:composition maps restricted to closed subschemes} and \cref{lemma: divisor iso from scheme theoretic inverse image}. Case (c) follows from the fact that the image of $\gamma$ is $T_{d,M}(q^{-1}(1), \dots, q^{-1}(n))$ (\cref{prop:comp-morph-iso-onto-closed-subscheme}) which is disjoint with $T_{d,M}(I)$ since $T_{d,M}(I,M_1,\dots, M_n)$ is empty by the nestedness criterion.
\end{proof}
This alone is almost enough to extend the operad composition morphisms to maps of corresponding log-schemes but we must also express 
$\gamma^* \mcl O_{T_{d,M}}(I)$ as a tensor product of the line bundles on $T_{d,n}\times \prod_{i=1}^n T_{d,q^{-1}(i)}$ for $I= M$, $I= q^{-1}(r)$, and $\abs{I} = 1$. Note that in all of these cases the pullback of the sections of these corresponding line bundles are all $0$.
To this end, recall that that \cref{lemma: M_I is tensor prod of O(I') duals} gives isomorphisms of line bundles on $T_{d,n}$ 
$$\mcl M_I \xra{\cong} \bigotimes_{I \sub I'} \mcl O_{T_{d,n}}(I')^\vee,$$
where $\mcl M_I$ are the line bundles of the universal simple screen on $T_{d,n}$ (\cref{subsec: Tdn def and properties}),
such that the following diagram commutes for every $I\sub J$,
\begin{center}
    \begin{tikzcd}
        \mcl M_I \arrow[r] \arrow[d]& \mcl M_J\arrow[d]\\
        \bigotimes_{I\sub I'} \mcl O_{T_{d,n}}(I')^\vee \arrow[r]& \bigotimes_{J \sub J'} \mcl O_{T_{d,n}}(J')^\vee
    \end{tikzcd}
\end{center}
where the bottom arrow is the map 
$$\bigotimes_{I\sub I',\ J\not \sub I'} s_{I'}^\vee.$$
In particular, this means that we have isomorphisms 
$$\mcl O_{T_{d,n}}(I) \cong \mcl M_I^\vee \bigotimes_{I \subsetneq I'} \mcl O_{T_{d,n}}(I)^\vee.$$

\begin{lemma}\label{lemma: pullbacks of LB via composition maps hard cases}
    There are canonical isomorphisms of line bundles
    \begin{enumerate}[label = (\alph*)]
        \item $\gamma^*\mcl O_{T_{d,M}}(M) \xra{\cong} \pi_0^*\mcl O_{T_{d,n}}([n])$.
        \item $\gamma^*\mcl O_{T_{d,M}}(q^{-1}(r)) \xra{\cong} \pi_0^*\mcl O_{T_{d,n}}(\{r\}) \otimes \pi_r^*\mcl O_{T_{d,q^{-1}(r)}}(q^{-1}(r))$.
        \item $\gamma^*\mcl O_{T_{d,M}}(\{i\}) \xra{\cong} \pi_r^*\mcl O_{T_{d,q^{-1}(r)}}(\{i\})$, if $i\in q^{-1}(r)$.
    \end{enumerate}
\end{lemma}
\begin{proof}
    First note that by the above we have $\mcl O_{T_{d,M}}(M) \cong \mcl (M_{M}^{d,M})^\vee$. By \cref{lemma: pullback of universal screen via composition map} we have $\gamma^*\mcl M_{M}^{d,M} \cong \pi_0^*\mcl M_{[n]}^{d,n}$. From this part (a) follows. 
    For part (b) note that \cref{lemma: pullbacks of LB via composition maps easy cases} gives an isomorphism
    $$\bigotimes_{q^{-1}(r)\subsetneq S} \gamma^*\mcl O_{T_{d,M}}(I) \cong \gamma^*\mcl O_{T_{d,M}}(M)\bigotimes_{r\in J \subsetneq [n]} \pi_0^*\mcl O_{T_{d,n}}(J) \cong \bigotimes_{r\in J \sub [n]} \pi_0^*\mcl O_{T_{d,n}}(J)$$
    where the last isomorphism follows from part (a). By definition, this is just $\pi_0^*\mcl O_{T_{d,n}}(\{r\})^\vee$
    and so we get an isomorphism
    $$\gamma^* \mcl O_{T_{d,M}}(q^{-1}(r)) \xra{\cong} \gamma^*(\mcl M_{q^{-1}(r)}^{d,M})^\vee \bigotimes_{q^{-1}(r)\subsetneq I'} \mcl \gamma^*O_{T_{d,M}}(I')^\vee\xra{\cong}\pi_r^*\mcl O_{T_{d,q^{-1}(r)}}(q^{-1}(r))\otimes \pi_0^*\mcl O_{T_{d,q^{-1}(r)}}(\{r\}).$$
    This gives part (b). Part (c) follows from a similar argument.
\end{proof}
\begin{remark}
    If $n = 1$ the left hand sides of (a) and (b) are equal. In this case, both statements are true but we must choose (b) to be the isomorphism defining our map of log schemes for the operad axioms to be satisfied. Similarly, if $\abs{q^{-1}(r)} = 1$ we must also choose the map in (b).
\end{remark}

With this we have defined isomorphisms of line bundles extending the operad composition maps of $\gk_d$ to morphisms of log-schemes
$$\gamma^{\log} \co \msf T_{d,n}\times \prod_{i=1}^n \msf T_{d,q^{-1}(i)} \to \msf T_{d,M}.$$
Unfortunately, there are no maps of log schemes
$$\spec \k \to (\spec \k ,\mcl O_\k \xra{0} \mcl O_\k)$$
and therefore we cannot extend our unit map to a map of log schemes. However, it is easy to verify that despite this the $\circ_i$-operations, $T_{d,n}\times T_{d,m} \xra{\circ_i} T_{d,m+n-1}$, do extend to maps of log schemes $\msf T_{d,n}\times \msf T_{d,m} \xra{\circ_i} \msf T_{d,m+n-1}$. One can verify that these maps, together with the symmetry morphisms satisfy the pseudo-operad axioms. 
\begin{definition}
    The \emph{log-geometric Chen, Gibney, Krashen pseudo-operad}, $\lgk_d$ is the pseudo-operad with objects $\msf T_{d,n}$ and with morphisms as described above.
\end{definition}
\begin{remark}
    In the case $d=1$ this is isomorphic to the framed little curves pseudo-operad constructed by Vaintrob in \cite{vaintrob2021formality}.
\end{remark}

\subsection{Kato--Nakayama Analytifications}
In this section we show that the Kato--Nakayama Analytification of $\lgk_d$ is homeomorphic to the $\S^1$-framed Fulton--MacPherson operad in dimension $2d$, $\tk_{2d}\rtimes \S^1$, which in turn is a model for the $\S^1$-framed little $2d$-disks operad $\ld_{2d} \rtimes \S^1$ as mentioned in \cref{sec:topological constructions}.

We begin by showing that there is a homeomorphism of spaces over $X^n(\C)$,
$$\lfm_n(X)^\kn \cong \topfm_n(X(\C)),$$
which commutes with the corresponding inclusions of $\conf_n(X(\C))$. The rough idea behind our proof of this statement is that the left hand side is the Kato--Nakayama space of a sequence of \emph{logarithmic blow-ups} of $X^n$ and the right hand side is a sequence of real blow-ups of $X(\C)^n$ and under sufficiently nice circumstances these two agree.

\begin{definition}
    Let $\msf X = (X, \mfr L)$ be a log scheme and let $Y$ be a closed subscheme of $X$. Then we let $\logbu{Y} \msf X$ denote the log scheme with underlying scheme $\bu{Y} X$ and log structure given by the log structure $\mfr L$ pulled back to $\bu{Y} X$ \emph{and} the line bundle with section corresponding to the exceptional divisor of the blow-up. We also define the \emph{blow-down map}, $\logbu{Y} \msf X \to X$ to be given by the normal blow-down map on underlying schemes and by ``forgetting'' the added line bundle with section.
\end{definition}

\begin{proposition}\label{prop: analytification of log blow-up}
    Let $X$ be a smooth $\C$-variety and let $Y$ be a smooth subvariety of codimension $k$.
    If $Y$ is the zero locus of a section of a rank $k$ vector bundle on $X$ then there is an isomorphism of spaces over $X(\C)$
    $$(\logbu{Y}X)^{\kn}\cong \realbu{Y(\C)} X(\C).$$
\end{proposition}
\begin{proof}
    This isomorphism is easy to construct when $Y$ is a complete intersection, i.e. the zero locus of a section of $\mcl O^{\oplus k}_X$. The general case follows by gluing. 
\end{proof}
\begin{corollary}\label{cor:log-bu is real-bu when nice}
    Let $\msf X = (X,\mfr L)$ be a log-scheme whose log structure is given by line bundles with sections cutting out smooth divisors $D_1,\dots, D_n$ such that $D_1\cup \dots \cup D_n$ is a strict normal crossings divisor and let $Y$ be as in the proposition with the additional constraint that $Y$ intersects all intersections $D_{i_1}\cap \dots D_{i_k}$ transversally. Then there is an isomorphism of spaces over $\msf X^{\kn}$
    $$(\logbu{Y}X)^{\kn}\cong \realbu{\tilde Y} \msf X^{\kn}$$
    where $\tilde Y$ denotes the strict transform of $Y(\C)$ under the sequence of blow-down maps 
    $$\msf X^{\kn} \to X(\C).$$
\end{corollary}
\begin{proof}
    Let $t\co \mcl O_X \to \mcl V$ denote the vector bundle section cutting out $Y$. It is clear from the proposition that there is an isomorphism 
    $$(\logbu{Y}X)^{\kn}\cong \realbu{\tilde t} \msf X^{\kn}$$
    where $\tilde t$ denotes the pullback (i.e. \emph{total} transform) to $\msf X^\kn$ of the vector bundle section $t$ seen as a section of a topological vector bundle $t \in \Gamma(X(\C),{\mcl V})$. The transversallity conditions for intersections with $Y$ imply that $\tilde Y$ is the zero locus of the section $\tilde t$. This completes the proof.
\end{proof}

\begin{proposition}
    There is a unique homeomorphism of spaces over $X^n(\C)$,
    $$\lfm_n(X)^\kn \cong \topfm_n(X(\C)),$$
    which commutes with the corresponding inclusions of $\conf_n(X(\C))$.
\end{proposition}
\begin{proof}
    First, recall from \cite{li2009wonderful} that the Fulton--MacPherson space $\afm_n(X)$ is isomorphic to the sequence of blow-ups
    $$\bu{\tilde \Delta_{I_N}}\dots \bu{\tilde \Delta_{I_1}} X^n$$ where $I_1, \dots, I_N$ are the sets in $\psets{[n]}$ in \emph{any} order satisfying the conditions of Theorem 1.3 in \cite{li2009wonderful} and $\tilde \Delta_{I_i}$ denotes the \emph{dominant} transform of the diagonal $\Delta_{I_i} \sub X^n$ under all previous blow-ups. Furthermore, recall that the divisors $D(I_i)$ are the \emph{dominant} transforms of the exceptional divisor $\tilde D(I_i)$ from the blow-up 
    $$\bu{\tilde \Delta_{I_i}}\bu{\tilde \Delta_{I_{i-1}}}\dots \bu{\tilde \Delta_{I_1}} X^n \to \bu{\tilde \Delta_{I_{i-1}}}\dots \bu{\tilde \Delta_{I_1}} X^n$$
    taken under all subsequent blow-ups $\bu{\tilde \Delta_{I_N}}\dots\bu{\tilde \Delta_{I_{i+1}}}$. Additionally, the topological Fulton--MacPherson space $\topfm_n(X(\C))$ can also be written as a sequence of real oriented blow-ups of $X^n(\C)$ in its diagonals in a similar manner.

    The proposition would now follow from \cref{cor:log-bu is real-bu when nice} by induction if there was an ordering $I_1, \dots, I_N$ of the sets in $\psets{[n]}$ satisfying the aforementioned conditions such that for every $1\leq i \leq n$ 
    $$\tilde \Delta_{I_{i+1}} \sub \bu{\tilde \Delta_{I_i}}\bu{\tilde \Delta_{I_{i-1}}}\dots \bu{\tilde \Delta_{I_1}} X^n$$ meets any intersection of the divisors $\tilde D(I_1) \dots \tilde D(I_i)$ transversally. It is not hard to show that any ordering where the sets $I\in \psets{[n]}$ appear in decreasing order of size, i.e. $\abs{I_i}\geq \abs{I_{i+1}}$, satisfies all criteria and thus we are done. 
\end{proof}
\begin{remark}
    Note that the transversality conditions here are not implied by the arrangement conditions in Theorem 1.3 of \cite{li2009wonderful}. Hence, there is no obvious analogue of the above proposition for general wonderful compactifications.
\end{remark}
\begin{corollary}\label{cor: K and T KN analytifications}
    There are isomorphisms 
    \begin{itemize}
        \item $\lfm_n(X)^\kn \cong \topfm_n(X(\C))$
        \item $\msf K_{d,n}^\kn \cong \kont_{2d,n}$
        \item $\msf T_{d,n} \cong (\S^1)^n\times \kont_{2d,n}$
    \end{itemize}
    which commute with all relevant maps, i.e. inclusions of $\conf_n(X(\C))$, maps to $X^n(\C)$, inclusions of fibres, etc. 
\end{corollary}
\begin{proof}
    We have already constructed the isomorphism $\lfm_n(X)^\kn \cong \topfm_n(X(\C))$. Since the morphism $\msf K_{d,n} \into \lfm_n(\A^d)$ is strict, the top square in the following diagram is Cartesian.
    \begin{center}
        \begin{tikzcd}
            \msf K_{d,n}^{\kn} \arrow[r] \arrow[d]& (\lfm_n(\A^d))^{\kn}\arrow[d]\\
            T_{d,n}(\C) \arrow[r] \arrow[d] &\afm_n(\A^d)(\C)\arrow[d]\\
            * \arrow[r, hook, "\textnormal{origin}"] & (\C^d)^n
        \end{tikzcd}
    \end{center}
    We already know that the bottom square is Cartesian and thus the entire diagram is Cartesian, i.e. $\msf K_{d,n}^\kn \to (\lfm_n(\A^d))^{\kn}$ is the inclusion of the fibre over the origin in $(\C^d)^n$ and thus $(\lfm_n(\A^d))^\kn \cong \topfm_n(\C^d)$ restricts to an isomorphism
    $\msf K_{d,n}^\kn \cong \kont_{2d,n}.$ Notice that we get $2d$ on the right hand side since $\C\cong \R^2$. Furthermore, we know that 
    $$\msf T_{d,n}^\kn = \realbu{0\in \mcl O(\{1\})}\dots\realbu{0\in \mcl O(\{n\})} \msf K_{d,n}^\kn.$$
    By \cref{cor: blow up in zero section of tensor prod of blown up line bundles is trivial} this is homeomorphic to 
    $(\S^1)^n\times \kont_{2d,n},$ i.e. the $n$th space in the $\tk_{2d}\rtimes \S^1$ operad. It is easy to show that these choices of isomorphisms make all relevant maps commute as claimed.
\end{proof}

Next, we will show that the maps $f_I\co \topfm_n(X(\C)) \to \realbu{\Delta_I}X^I$, $\pi_I\co \kont_{2d,n} \to \S^{2d(\abs{I}-1)-1}$ (which were defined in \cref{sec:topological constructions}) and the projections to the $\S^1$ components $\pi_i\co (\S^1)^n\times \kont_{2d,n} \to \S^1$ are all identified with analytifications of explicit maps of corresponding log schemes under the isomorphisms described in the proof of \cref{cor: K and T KN analytifications}. 

Let $\msf S_{m}$ denote the log scheme $(\P^m, 0\co \mcl O \to \mcl O(-1))$. Its Kato--Nakayama space $\msf S_m^{\kn}$ homeomorphic to $\S^{2m+1}$ and the map $\S^{2m+1}\cong S_m^{\kn} \to \P^m(\C)$ is the Hopf fibration (proving this using \cref{prop: analytification of log blow-up} may provide some insight into the idea behind the arguments to come).
Recall from \cref{subsec: alg fm} that the composition
$$\afm_n(X) \to X^n \to X^I$$
factors via a map $f_I \co \afm_n(X) \to \bu{\Delta_I} X^I$ for any $I \in\psets{n}$, and furthermore that there is an isomorphism of line bundles with sections 
$$f_I^*\mcl O(E_{\Delta_I}) \cong \bigotimes_{I \sub I'}\mcl O(D(I')).$$ This isomorphism defines a map of log schemes 
$$f_I^{\log}\co \lfm_n(X) \to \logbu{\Delta_I}X^I.$$

Also recall that the fibre over the origin for the blow-down map $\bu{\Delta_I} X^I \to X^I$ is $\P^{d(\abs{I}-1)-1}$ and that, if $i$ denotes the inclusion of this fibre, we have that $i^* \mcl O(E_{\Delta_I}) \cong \mcl O(-1)$.  
This means that 
$\msf S_{d(\abs{I}-1)-1} =  (\P^{d(\abs{I}-1)-1}, 0\co \mcl O \to \mcl O(-1))$ is the fibre over the origin of the log blow-down map $\logbu{\Delta_I} X^I \to X^I$ and $f_I^{\log}$ restricts to a map $\pi_I^{\log}\co \msf K_{d,n} \to \msf S_{d(\abs{I}-1)-1}$.

The maps we have described fit together in the following diagram:
\begin{center}
        \[\begin{tikzcd}
        	{\conf_n(\A^d)} && {\afm_n(\A^d)} && {\msf K_{d,n}} \\
        	\\
        	&& {\logbu{\Delta_{I}}(\A^d)^{I}} && {\msf S_{d(|S|-1)-1}} \\
        	\\
        	&& {(\A^d)^{I}}
        	\arrow[from=1-1, to=1-3]
        	\arrow[from=1-1, to=5-3]
        	\arrow["{f_I^{\log}}", from=1-3, to=3-3]
        	\arrow[from=1-5, to=1-3]
        	\arrow["{\pi_I^{\log}}", from=1-5, to=3-5]
        	\arrow[from=3-3, to=5-3]
        	\arrow[from=3-5, to=3-3]
        	\arrow["0", from=3-5, to=5-3]
        \end{tikzcd}\]
    \end{center}
The analytification of this diagram is:
\begin{center}
        \[\begin{tikzcd}
        	{\conf_n(\R^{2d})} && {\topfm_n(\R^{2d})} && {\kont_{2d,n}} \\
        	\\
        	&& {\realbu{\Delta_{I}}(\R^{2d})^{I}} && {\S^{2d(|S|-1)-1}} \\
        	\\
        	&& {(\R^{2d})^{I}}
        	\arrow[from=1-1, to=1-3]
        	\arrow[from=1-1, to=5-3]
        	\arrow["{(f_I^{\log})^\kn}", from=1-3, to=3-3]
        	\arrow[from=1-5, to=1-3]
        	\arrow["{(\pi_I^{\log})^\kn}", from=1-5, to=3-5]
        	\arrow[from=3-3, to=5-3]
        	\arrow[from=3-5, to=3-3]
        	\arrow["0", from=3-5, to=5-3]
        \end{tikzcd}\]
    \end{center}
Since $\conf_n(\R^{2d})$ is dense in $\topfm_n(\R^{2d})$ and since the blow-down map $\rho_I\co \realbu{\Delta_{I}}(\R^{2d})^{I} \to (\R^{2d})^I$ has an inverse on $\conf_I(\R^{2d})$ there is only one possible set of maps $(f_I^{\log})^\kn$ and $(\pi_I^{\log})^\kn$ which can make this diagram commute. By definition, this diagram also commutes if we insert the maps $f_I, \pi_I$ defined in \cref{sec:topological constructions} and thus we must have $(f_I^{\log})^\kn = f_I$ and $(\pi_I^{\log})^\kn = \pi_I$.

Finally, let $\pi_i^{\log}$ be the map
$$\pi_i^{\log}\co \msf T_{d,n} \to \msf S_0 = (\spec \C, 0\co \mcl O_\C\to \mcl O_\C)$$ given by the only possible map on the level of schemes and the isomorphism of log structures 
$$(\pi_i^{\log})^* \mcl O_{\C} \xra{\cong} \bigotimes_{i\in S} \mcl O_{T_{d,n}}(I).$$
It is clear that the analytification of this map is the projection map $\pi_i\co (\S^1)^n\times \kont_{2d,n} \to \S^1$ as desired.

We are now ready to show that the analytifications of the composition and symmetry maps in $\lgk_d$ are the composition and symmetry maps of $\tk_{2d}\rtimes \S^1$. For the symmetry action this is trivial as we can again use that $\conf_n(\C^d) \sub \topfm_n(\C^d)$ is dense to conclude this using an argument similar to the one above. 

A somewhat more complicated argument is needed to prove that the analytifications of the $\lgk_n$ composition morphisms and $\circ_i$-operations are the corresponding maps of the $\tk_{2d}\rtimes \S^1$ operad. We merely give an outline of this argument. 
First note that for any surjection of finite sets $\vp\co I\quot J$ the induced map $(\A^d)^J \to (\A^d)^I$ extends uniquely to a map $\logbu{\Delta_J}(\A^d)^J \to \logbu{\Delta_I}(\A^d)^I$ which restricts to a map $g_\vp^{\log}\co \msf S_{d(\abs{J}-1)-1}\to \msf S_{d(\abs{I}-1)-1}$. By the same argument as above the analytification of this map must be the map $g_\vp\co \S^{2d(\abs{J}-1)-1} \to \S^{2d(\abs{I}-1)-1}$ defined in \cref{sec:topological constructions}. Using this one can verify that the analytification of $\gamma^{\log}$ is the composition morphism of $\tk_{2d}\rtimes \S^1$ componentwise. That is, if $\gamma$ denotes the composition map in $\tk_{2d}\rtimes \S^1$ we can verify that 
$$\pi_I \circ \gamma = \pi_I\circ (\gamma^{\log})^\kn = (\pi_I^{\log}\circ\gamma^{\log})^\kn$$
for every $I \sub M$ and from this $\gamma = (\gamma^{\log})^\kn$ immediately follows. The $\circ_i$ case is analogous.
We include the proof for one important step which hopefully also provides some insight into the methods used to fill in the remaining steps of the argument.

Let $f\co \msf S_0\times \msf S_n\to \msf S_n$ be the map given by the canonical isomorphism $f\co \spec \C\times \P^n\to \P^n$ on underlying schemes and by the obvious isomorphism $f^*\mcl O_{\P^n}(-1) \xra{\cong} \pi_1^*\mcl O_{\spec \C}\otimes \pi_2^*\mcl O_{\P^n}(-1)$ for the log structures. Here $\pi_1,\pi_2$ denote the corresponding projection maps from the product $\spec \C \times \P^n$. 
\begin{proposition}\label{prop: S^1 action on S^(2n-1) is algebraic}
    The analytification $f^\kn\co \S^1\times \S^{2n-1} \to \S^{2n-1}$ is the $\S^1\cong \SO(2)$ action on $\S^{2n-1}$ induced by the diagonal inclusion $\SO(2) \into \SO(2n)$.
\end{proposition}
\begin{proof}
    First, let $F\co \A^1 \times \A^n \to \A^n$ be the map $(z,(x_1,\dots, x_n))\mapsto (zx_1,\dots, zx_n)$. By the universal property of the blow-up this induces a map 
    $\tilde F \co \A^1\times \bu{p}\A^n \to \bu{p} \A^n$, where $p$ denotes the origin in $\A^n$. Furthermore, $\tilde F^{-1}(E_p)$, where $E_p$ is the exceptional divisor in $\bu{p}\A^n$, is the scheme theoretic union $o\times\bu{p}\A^n \cup \A^1 \times E_p$. By lemma \ref{lemma: divisor iso from scheme theoretic inverse image} this gives a commutative diagram of log varieties:
    \begin{center}
        \begin{tikzcd}
            \msf S_0\times \msf S_{n-1} \arrow{d}{i} \arrow{rr}{f}&& \msf S_0\arrow{d}{j}\\
            \logbu{o}\A^1 \times \logbu{p}\A^n \arrow{rr}{\tilde F^{\log}}\arrow{d}&& \logbu{p}\A^n\arrow{d}{\rho}\\
            \A^1\times \A^n \arrow{rr}{F}&& \A^n
        \end{tikzcd}
    \end{center}
    where $\tilde F^{\log}$ denotes the map $\tilde F$ extended to a map log varieties in the obvious way and where $i,j$ are the strict morphisms induced by the inclusions of fibres over the origin for the underlying schemes. The analytification of this diagram is of the following form:
    \begin{center}
        \begin{tikzcd}
            \S^1 \times \S^{2n-1} \arrow{d} \arrow{rr}{f'}&& \S^{2n-1}\arrow{d}\\
            \realbu{o} \C \times \realbu{p}\C^{n} \arrow{rr}{\tilde F'}\arrow{d}&& \realbu{p}\C^{n}\arrow{d}\\
            \C\times \C^{n} \arrow{rr}{F^\an} &&\C^{n}
        \end{tikzcd}
    \end{center}

    where $\tilde F' = \tilde F^\kn$ and $f' = f^\kn$. Since $\rho$ is a homeomorphism on a dense subset there is only one pair $f',\tilde F'$ which can make the diagram commute by the same argument as above. It is easy to construct a function $\tilde F'$ such that the diagram commutes with $f'$ equal to the group action morphism and from this the lemma follows.
\end{proof}

The results of this section taken together give our main theorem.
\begin{theorem}\label{thm: main theorem (non-unital operad iso)}
    The Kato--Nakayama analytification of $\lgk_d$ is homeomorphic to $\tk_{2d}\rtimes \S^1$.
\end{theorem}

\subsection{Virtual Log Geometric Generalizations}\label{subsec: virtual log operads}
We end this article by discussing some apparent generalizations of these results in the category of virtual log schemes. The most immediate improvement we get with virtual morphisms is that, by \cref{example: virt morph point into circle}, we there are virtual morphisms 
$$\spec \C \to \msf S_0.$$
It is not hard to show that we can choose such a morphism such that it satisfies the unit axiom for our operad. Hence, in the category of log schemes with virtual morphisms we can define $\lgk_d$ to be an operad, not just a pseudo-operad. It is clear that \cref{thm: main theorem (non-unital operad iso)} still holds for this operad. 

Virtual morphisms also allow us to construct the (unframed) Fulton--MacPherson operad $\tk_{2d}$ as the analytification of an operad of log schemes with virtual morphisms. 
Indeed, the operad composition map
$$\gamma \co T_{d,n}\times \prod_{i=1}^n T_{d,q^{-1}(i)} \to T_{d,M}$$
extends to a virtual morphism of log schemes
$$\msf K_{d,n}\times \prod_{i=1}^n \msf K_{d,q^{-1}(i)} \to \msf K_{d,M}.$$
This map is defined using \cref{lemma: pullbacks of LB via composition maps easy cases} and \cref{lemma: pullbacks of LB via composition maps hard cases} but this time we replace the isomorphism in \cref{lemma: pullbacks of LB via composition maps hard cases}, (b) with the isomorphism
$$\gamma^*\mcl O_{T_{d,M}}(q^{-1}(r)) \xra{\cong} \left (\bigotimes_{r\in I}\pi_0^* \mcl O_{T_{d,n}}(I)^{\otimes(-1)}\right ) \otimes \pi_r^* \mcl O_{d, q^{-1}(r)}(q^{-1}(r)).$$
This gives a well defined virtual morphism of log structures by our definitions. These maps define an operad $\vgk_d$ in log schemes with virtual morphisms.
Using the same methods as before one can show that the analytification of the resulting operad is $\tk_{2d}$.
\begin{theorem}\label{thm: virtual log operad}
        The Kato--Nakayama analytification of $\vgk_d$ is homeomorphic to $\tk_{2d}$.
\end{theorem}

\bibliographystyle{alphaurl}
\bibliography{references}

\begin{thebibliography}{BDPW23}

\bibitem[AK10]{Arone2010OnTF}
Gregory Arone and Marja Kankaanrinta.
\newblock On the functoriality of the blow-up construction.
\newblock {\em Bull. Belg. Math. Soc. Simon Stevin}, 17(5):821--832, 2010.
\newblock URL: \url{http://projecteuclid.org/euclid.bbms/1292334057}.

\bibitem[AS94]{axelrodsinger1994chernsimons}
Scott Axelrod and I.~M. Singer.
\newblock Chern-{S}imons perturbation theory. {II}.
\newblock In {\em Perspectives in mathematical physics}, volume III of {\em
  Conf. Proc. Lecture Notes Math. Phys.}, pages 17--49. Int. Press, Cambridge,
  MA, 1994.

\bibitem[BDPW23]{bergstrom2023hyperelliptic}
Jonas Bergstr{\"o}m, Adrian Diaconu, Dan Petersen, and Craig Westerland.
\newblock Hyperelliptic curves, the scanning map, and moments of families of
  quadratic {L}-functions.
\newblock {\em arXiv preprint arXiv:2302.07664}, 2023.

\bibitem[BN98]{bar-natan1998associators}
Dror Bar-Natan.
\newblock On associators and the {G}rothendieck-{T}eichmuller group. {I}.
\newblock {\em Selecta Math. (N.S.)}, 4(2):183--212, 1998.
\newblock \href {https://doi.org/10.1007/s000290050029}
  {\path{doi:10.1007/s000290050029}}.

\bibitem[CGK09]{chen2009pointed}
L.~Chen, A.~Gibney, and D.~Krashen.
\newblock Pointed trees of projective spaces.
\newblock {\em J. Algebraic Geom.}, 18(3):477--509, 2009.
\newblock \href {https://doi.org/10.1090/S1056-3911-08-00494-3}
  {\path{doi:10.1090/S1056-3911-08-00494-3}}.

\bibitem[dBHK24]{debrito2024algebrogeometricmodelconfigurationcategory}
Pedro~Boavida de~Brito, Geoffroy Horel, and Danica Kosanović.
\newblock An algebro-geometric model for the configuration category, 2024.
\newblock URL: \url{https://arxiv.org/abs/2411.06934}, \href
  {https://arxiv.org/abs/2411.06934} {\path{arXiv:2411.06934}}.

\bibitem[DPP24]{dupont2024logarithmicmorphismstangentialbasepoints}
Clément Dupont, Erik Panzer, and Brent Pym.
\newblock Logarithmic morphisms, tangential basepoints, and little disks, 2024.
\newblock URL: \url{https://arxiv.org/abs/2408.13108}, \href
  {https://arxiv.org/abs/2408.13108} {\path{arXiv:2408.13108}}.

\bibitem[Dri90]{drinfeld1990quasitriangular}
V.~G. Drinfeld.
\newblock On quasitriangular quasi-{H}opf algebras and on a group that is
  closely connected with {${\rm Gal}(\overline{\bf Q}/{\bf Q})$}.
\newblock {\em Algebra i Analiz}, 2(4):149--181, 1990.

\bibitem[FM94]{fulton1994acompactification}
William Fulton and Robert MacPherson.
\newblock A compactification of configuration spaces.
\newblock {\em Annals of Mathematics}, 139(1):183--225, 1994.

\bibitem[GJ94]{getzler1994operadshomotopyalgebraiterated}
Ezra Getzler and J.~D.~S. Jones.
\newblock Operads, homotopy algebra and iterated integrals for double loop
  spaces, 1994.
\newblock URL: \url{https://arxiv.org/abs/hep-th/9403055}, \href
  {https://arxiv.org/abs/hep-th/9403055} {\path{arXiv:hep-th/9403055}}.

\bibitem[GK98]{getzler98modular}
E.~Getzler and M.~M. Kapranov.
\newblock Modular operads.
\newblock {\em Compositio Math.}, 110(1):65--126, 1998.
\newblock \href {https://doi.org/10.1023/A:1000245600345}
  {\path{doi:10.1023/A:1000245600345}}.

\bibitem[{How}17]{howell2017motives}
Nicholas~L. {Howell}.
\newblock {\em {Motives of Log Schemes}}.
\newblock PhD thesis, University of Oregon, January 2017.

\bibitem[KN99]{kato1999log}
Kazuya Kato and Chikara Nakayama.
\newblock Log {B}etti cohomology, log \'etale cohomology, and log de {R}ham
  cohomology of log schemes over {${\bf C}$}.
\newblock {\em Kodai Math. J.}, 22(2):161--186, 1999.
\newblock \href {https://doi.org/10.2996/kmj/1138044041}
  {\path{doi:10.2996/kmj/1138044041}}.

\bibitem[Kon94]{kontsevich1994feynman}
Maxim Kontsevich.
\newblock Feynman diagrams and low-dimensional topology.
\newblock In {\em First {E}uropean {C}ongress of {M}athematics, {V}ol.\ {II}
  ({P}aris, 1992)}, volume 120 of {\em Progr. Math.}, pages 97--121.
  Birkh\"auser, Basel, 1994.

\bibitem[Kon99]{kontsevich1999operads}
Maxim Kontsevich.
\newblock Operads and motives in deformation quantization.
\newblock {\em Lett. Math. Phys.}, 48(1):35--72, 1999.
\newblock Mosh\'e{} Flato (1937--1998).
\newblock \href {https://doi.org/10.1023/A:1007555725247}
  {\path{doi:10.1023/A:1007555725247}}.

\bibitem[KW25]{khoroshkin2025realmodelsframedlittle}
Anton Khoroshkin and Thomas Willwacher.
\newblock Real models for the framed little $n$-disks operads, 2025.
\newblock URL: \url{https://arxiv.org/abs/1705.08108}, \href
  {https://arxiv.org/abs/1705.08108} {\path{arXiv:1705.08108}}.

\bibitem[Li09]{li2009wonderful}
Li~Li.
\newblock Wonderful compactification of an arrangement of subvarieties.
\newblock {\em Michigan mathematical journal}, 58(2):535--563, 2009.

\bibitem[Lin24]{lindström2024logarithmic}
Oliver Lindström.
\newblock Logarithmic geometry and the {S}1-framed {K}ontsevich operad.
\newblock 2024.

\bibitem[LV14]{lambrecths2014formality}
Pascal Lambrechts and Ismar Voli\'c.
\newblock Formality of the little {$N$}-disks operad.
\newblock {\em Mem. Amer. Math. Soc.}, 230(1079):viii+116, 2014.

\bibitem[Mor07]{morava2003motivicthomisomorphism}
Jack Morava.
\newblock The motivic {T}hom isomorphism.
\newblock In {\em Elliptic cohomology}, volume 342 of {\em London Math. Soc.
  Lecture Note Ser.}, pages 265--285. Cambridge Univ. Press, Cambridge, 2007.
\newblock \href {https://doi.org/10.1017/CBO9780511721489.014}
  {\path{doi:10.1017/CBO9780511721489.014}}.

\bibitem[MSS07]{markl2007operads}
Martin Markl, James~D. Stasheff, and S.~Shnider.
\newblock {\em Operads in Algebra, topology and physics}.
\newblock American Mathematical Society, 2007.

\bibitem[Pet14]{petersen2013minimalmodelsgtactionformality}
Dan Petersen.
\newblock Minimal models, {GT}-action and formality of the little disk operad.
\newblock {\em Selecta Math. (N.S.)}, 20(3):817--822, 2014.
\newblock \href {https://doi.org/10.1007/s00029-013-0135-5}
  {\path{doi:10.1007/s00029-013-0135-5}}.

\bibitem[Sal01]{salvatore1999configuration}
Paolo Salvatore.
\newblock Configuration spaces with summable labels.
\newblock In {\em Cohomological methods in homotopy theory ({B}ellaterra,
  1998)}, volume 196 of {\em Progr. Math.}, pages 375--395. Birkh\"auser,
  Basel, 2001.

\bibitem[Sin04]{sinha2004manifold}
Dev~P. Sinha.
\newblock Manifold-theoretic compactifications of configuration spaces.
\newblock {\em Selecta Math. (N.S.)}, 10(3):391--428, 2004.
\newblock \href {https://doi.org/10.1007/s00029-004-0381-7}
  {\path{doi:10.1007/s00029-004-0381-7}}.

\bibitem[SW03]{salvatore2003framed}
Paolo Salvatore and Nathalie Wahl.
\newblock Framed discs operads and {B}atalin-{V}ilkovisky algebras.
\newblock {\em Q. J. Math.}, 54(2):213--231, 2003.
\newblock \href {https://doi.org/10.1093/qjmath/54.2.213}
  {\path{doi:10.1093/qjmath/54.2.213}}.

\bibitem[Vai19]{vaintrob2019moduliframedformalcurves}
Dmitry Vaintrob.
\newblock Moduli of framed formal curves, 2019.
\newblock URL: \url{https://arxiv.org/abs/1910.11550}, \href
  {https://arxiv.org/abs/1910.11550} {\path{arXiv:1910.11550}}.

\bibitem[Vai21]{vaintrob2021formality}
Dmitry Vaintrob.
\newblock Formality of little disks and algebraic geometry, 2021.
\newblock URL: \url{https://arxiv.org/abs/2103.15054}, \href
  {https://arxiv.org/abs/2103.15054} {\path{arXiv:2103.15054}}.

\end{thebibliography}
\end{document}